\documentclass[12pt]{amsart}

\usepackage{amsmath,latexsym,amssymb,verbatim,euscript,graphicx}
\usepackage{fullpage}
\usepackage[text={6.5in,9in},centering]{geometry}
\usepackage{mathabx}
\usepackage[thinlines]{easymat}
\newcommand{\R}{\mathbb{R}}

\newcommand{\II}{\operatorname{II}}

\newcommand{\e}{\mathbf{e}}
\newcommand{\bfh}{\mathbf{h}}
\newcommand{\bfu}{\mathbf{u}}
\newcommand{\bff}{\mathbf{f}}
\newcommand{\bfv}{\mathbf{v}}
\newcommand{\bfw}{\mathbf{w}}
\newcommand{\bfx}{\mathbf{x}}
\newcommand{\bfy}{\mathbf{y}}
\newcommand{\bfz}{\mathbf{z}}
\newcommand{\scE}{{\EuScript E}}

\newtheorem{ourtheorem}{Theorem}
\newtheorem{theorem}{Theorem}[section]
\newtheorem{non-theorem}{Non-Theorem}[section]
\newtheorem{lemma}[theorem]{Lemma}

\newtheorem{proposition}[theorem]{Proposition}
\newtheorem{theoremvoid}{Theorem} 
{\theoremstyle{definition}
\newtheorem{example}[theorem]{Example}
\newtheorem{definition}[theorem]{Definition}
\newtheorem{remark}[theorem]{Remark}
\newtheorem{assumption}[theorem]{Assumption}
\newtheorem{notation}[theorem]{Notation}

\setcounter{MaxMatrixCols}{20}

\numberwithin{equation}{section}

\begin{document}

\title[Isometric embedding]{Isometric embedding via strongly symmetric positive systems}

\author{Gui-Qiang Chen}
\address{Mathematical Institute, University of Oxford, Oxford, OX2 6GG, UK}
\email{chengq@maths.ox.ac.uk}

\author{Jeanne Clelland}
\address{Department of Mathematics, 395 UCB, University of
Colorado, Boulder, CO 80309-0395, USA}
\email{Jeanne.Clelland@colorado.edu}

\author{Marshall Slemrod}
\address{Department of Mathematics, University of Wisconsin, Madison, WI 53706, USA}
\email{slemrod@math.wisc.edu}

\author{Dehua Wang}
\address{Department of Mathematics, University of Pittsburgh, Pittsburgh, PA 15260, USA}
\email{dwang@math.pitt.edu}

\author{Deane Yang}
\address{Department of Mathematics, New York University, New York, NY 10012, USA}
\email{deane.yang@nyu.edu}

\subjclass[2010]{Primary(53B20, 53C42), Secondary(35F50)}
\keywords{isometric embedding, strongly symmetric positive systems}

\begin{abstract}
We give a new proof for the local existence of a smooth isometric embedding of a smooth $3$-dimensional Riemannian manifold with nonzero Riemannian curvature tensor
into $6$-dimensional Euclidean space.  Our proof avoids the sophisticated arguments via microlocal analysis used in earlier proofs.

In Part 1, we introduce a new type of system of partial differential equations (PDE), which is not one of the standard types (elliptic, hyperbolic, parabolic)
but satisfies a property called strong symmetric positivity.
Such a PDE system is a generalization of and has properties similar to a system of ordinary differential equations with a regular singular point.
A local existence theorem is then established by using a novel local-to-global-to-local approach.
In Part 2, we apply this theorem to prove the local existence result for isometric embeddings.
\end{abstract}

\maketitle

\tableofcontents

\section{Introduction}

Let $(M,g)$ be an $n$-dimensional $C^\infty$ Riemannian manifold.
Recall that a $C^\infty$ map $\bfy: M \rightarrow \R^N$ is called an {\em isometric embedding} if $\bfy$ is injective and
the restriction of the Euclidean metric on $\R^N$ to the image $\bfy(M)$ agrees with the metric $g$ on $M$.
In terms of local coordinates $\bfx = (x^1, \ldots, x^n)$ on $M$, this is equivalent to the condition that
\begin{equation}\label{equation:isometric}
\partial_i\bfy\cdot\partial_j\bfy = g_{ij}, \qquad 1 \leq i,j \leq n,
\end{equation}
where $g = g_{ij} dx^i dx^j$ and $\partial_i$ denotes $\frac{\partial}{\partial x^i}$.

In this paper, we study the local isometric embedding problem, which asks whether, given a Riemannian manifold $(M,g)$ and a point $\bfx_0 \in M$,
there exists an isometric embedding of some neighborhood  of $\bfx_0$ into $\R^N$---i.e., whether the PDE system \eqref{equation:isometric}
has local $C^\infty$ solutions in some neighborhood of $\bfx_0$.
The system \eqref{equation:isometric} consists of $\tfrac{1}{2} n (n+1)$ partial differential equations
for $N$ unknown functions $\bfy = (y^1, \ldots, y^N)$;
thus it is overdetermined when $N < \tfrac{1}{2} n (n+1)$, underdetermined when $N > \tfrac{1}{2} n (n+1)$, and determined when $N = \tfrac{1}{2} n (n+1)$.

The isometric embedding problem has a long and active history.
The famous theorem of Cartan and Janet (see, e.g., \cite{Jacobowitz82}) guarantees that,
when the metric $g$ is real analytic, local real analytic solutions to \eqref{equation:isometric}
always exist in the determined case $N = \tfrac{1}{2} n (n+1)$.
In the $C^\infty$ category, much less is known.  Nash \cite{Nash56} proved a global existence theorem
in the highly underdetermined case $N = \tfrac{1}{2}n (n+1) (3n+11)$.
Later, refinements were given by Greene \cite{Greene70} and Gunther \cite{Gunther89}
for the local existence problem  that improved the upper bound on the embedding dimension to $N = \tfrac{1}{2} n (n+1) + n$.

When $N = \tfrac{1}{2} n (n+1)$, known results for $g$ in the $C^\infty$ category are limited to $n\leq 4$.
Most research activity has been concentrated on the case $n=2$, where local isometric embeddings of
varying regularity have been shown to exist in a neighborhood of any point $\bfx_0 \in M$
where either the Gauss curvature $K(\bfx_0)$ is nonzero, $K(\bfx_0) = 0$ and $\nabla K(\bfx_0) \neq 0$,
or $K(\bfx_0)$ vanishes to finite order in certain precise ways (cf. \cite{Han05, Han06, HHL03, HK10, Khuri07, Lin85, Lin86}).
For a detailed account, see \cite{HH06}.

For $n \geq 3$, there are fewer results.
Bryant, Griffiths, and Yang \cite{BGY83} showed that, for $n=3$, local $C^\infty$ isometric embeddings 
exist in a neighborhood of any point $\bfx_0 \in M$ where the Einstein tensor has rank greater than $1$.
Subsequent work was able to relax this restriction on the Einstein tensor: In \cite{NM89},
Nakamura and Maeda extended the existence theorem to a neighborhood of any point where the Riemann curvature tensor does not vanish,
and in \cite{Poole10}, Poole extended the existence theorem to a neighborhood of any point
where the Riemann curvature tensor vanishes but its covariant derivative does not.
Meanwhile, for $n=4$, the results of \cite{BGY83}, \cite{GY88}, and \cite{NM89} imply that
there exists a finite set of algebraic relations among the Riemann curvature tensor and its covariant derivatives,
with the property that a local isometric embedding exists in a neighborhood of any point where these relations do not all hold.

Our main result is a new, simpler proof of the following theorem of Nakamura-Maeda \cite{NM89} when $n=3$ and $N=6$ (also see Goodman-Yang \cite{GY88}):

\begin{theoremvoid}[cf. Theorem \ref{main-theorem-3D}]
Let $(M,g)$ be a $C^\infty$ Riemannian manifold of dimension $3$; let $\bfx_0 \in M$ so that
the Riemann curvature tensor $R(\bfx_0)$ is nonzero.
Then there exists a neighborhood $\Omega \subset M$ of $\bfx_0$ for which there is a $C^\infty$ isometric embedding $\bfy:\Omega \to \R^6$.
\end{theoremvoid}

Our proof, like the previous ones, uses the Nash-Moser implicit function
theorem (cf. Theorem \ref{Nash-Moser-theorem}) to obtain a solution.  This requires showing that the
linearized system has a solution that satisfies certain estimates known as ``smooth tame estimates" (this terminology 
is due to Hamilton; see \cite{HamiltonIFT}).
The advantage of our approach is that it completely eliminates the need for the microlocal analysis and Fourier integral operators used
in the proofs of Nakamura-Maeda and Goodman-Yang; instead, it is based on Friedrichs's theory of symmetric positive systems.

Friedrichs \cite{Friedrichs58}
 introduced the notion of a symmetric positive partial
differential operator $P$ to study
a class of first order linear systems of PDEs
\begin{equation}\label{system}
P\bfv = A^i\partial_i\bfv + B\bfv = \mathbf{h}
\end{equation}
that do not necessarily fall into one of
the standard types (elliptic, hyperbolic, parabolic). He proved,
under suitable boundary conditions on the domain $\Omega$, the
existence and uniqueness of an $L^2(\Omega)$ solution to the system \eqref{system}.
No higher order regularity of solutions is guaranteed, even if the functions $A^i$, $B$, and $\mathbf{h}$ are $C^\infty$.

We call a domain that satisfies Friedrichs's boundary condition {\em {P}-convex} (cf. Definition \ref{P-convex-def}).
Such a domain $\Omega$ has the remarkable property that any solution $\bfv$ to a symmetric positive system \eqref{system}
on $\Omega$ is {\em unique} in $L^2(\Omega)$, {\em without assuming any boundary
conditions on $\bfv$.}  This surprising rigidity occurs because a symmetric positive operator $P$ always
has a subtle type of singularity in the interior of a $P$-convex domain.  In \S \ref{state-theorem-sec},
we give a $1$-dimensional example, where the system reduces to a scalar ODE, that illustrates how this occurs.

We introduce in this paper a new positivity condition that we call {\em strong symmetric positivity} (cf. Definition \ref{ssp-def})
and prove a local existence and regularity theorem for first order linear and nonlinear systems satisfying it (cf. Theorem \ref{nonlinear-theorem}).
As the name indicates, this condition is a strengthening of Friedrichs's notion of symmetric positivity.
Moser \cite{Moser1} introduced a similar but weaker assumption, closely related to the Legendre-Hadamard condition,
and proved that any real analytic system of the form \eqref{system} satisfying this condition on a $P$-convex domain has a unique real analytic solution $\bfv$.
Tso \cite{Tso92} proved a similar $C^\infty$ existence theorem on a $P$-convex domain under Moser's condition,
but we believe that his proof actually requires the stronger assumption of strong symmetric positivity.
Both Moser and Tso
used their results for linear systems to prove analogous
perturbation theorems for nonlinear strongly symmetric positive
systems
\begin{equation}\label{nonlinear}
\Phi(\bfu) = \bff
\end{equation}
on a domain $\Omega \subset \R^n$,
provided that $\bff$ is sufficiently close to $\Phi(\bfu_0)$ for a given function $\bfu_0$,
and $\Omega$ is $P$-convex, where $P$ is the
linearization of $\Phi$ at $\bfu_0$.

Our proof of Theorem \ref{main-theorem-3D} proceeds in two major steps. In Part 1
(\S \ref{ssp-sec}--\S \ref{proof-sec}), we establish the local solvability of a nonlinear strongly symmetric
positive system using the Nash-Moser implicit function
theorem. In Part 2 (\S \ref{state-main-theorems-sec}--\S \ref{ssp-sec-2}), we show that, if the Riemann
curvature tensor is nonzero at $\bfx_0 \in M$, then there exists an approximate
isometric embedding on a neighborhood of $\bfx_0$ where the linearized
operator can be made strongly symmetric positive by applying a carefully chosen change of
variables.  This argument consists primarily of linear algebra and requires essentially no analysis 
beyond that required for Part 1.  Theorem \ref{main-theorem-3D} then follows by the smooth tame estimates
established in Part 1 and the Nash-Moser implicit function theorem.

The first step requires solving linear strongly symmetric
positive systems on a sufficiently small, but fixed, neighborhood of a
point $\bfx_0$
in the domain and showing that solutions satisfy smooth
tame estimates.
Surprisingly, Tso's global existence theorem for strongly symmetric
positive systems on a $P$-convex domain does not directly imply a local
solvability theorem. This is because there does not necessarily exist
a $P$-convex domain in a neighborhood of a given point $\bfx_0$.
This subtle fact is
best illustrated by the $1$-dimensional example given in \S\ref{state-theorem-sec}.
In \S\ref{proof-sec}, we show how this difficulty may be overcome by first restricting the linearized
system to a sufficiently small neighborhood of $\bfx_0$ and then extending
the restricted system to a large ball in $\R^n$ that is $P$-convex for the extended system.

Before proceeding, we recall the following standard notations and facts regarding Sobolev spaces on a domain $\Omega \subset \R^n$:
\begin{itemize}
\item The Euclidean norm on vectors or matrices is denoted by $| \cdot |$,  and the $\ell_\infty$-norm on 
vectors or matrices is denoted by $| \cdot |_\infty$.
\item  The Sobolev spaces are denoted by
$$
W^{k,p}(\Omega)=\{u\in L^p(\Omega) \, :\,  \|u\|_{k, p}<\infty\},
$$
where  $\|u\|_{k, p}=\sum_{|\alpha|\le k} \|D^\alpha u\|_{L^p}$ is the Sobolev norm for the 
multi-index $\alpha=(\alpha_1$, $\ldots$, $\alpha_n)$, and $D^\alpha u=\frac{\partial^\alpha u}{(\partial x^1)^{\alpha_1}\cdots (\partial x^n)^{\alpha_n}}$.

\item For $p=2$, $W^{k,2}(\Omega)$ is denoted by $H^k(\Omega)$,
 with the norm $\|\cdot\|_{k,2}$ denoted by
$\|\cdot\|_k$.

\item The $C^k(\Omega)$-norm is denoted by
\[ 
\| u \|_{k,\infty} = \sum_{j=0}^k \sum_{|\alpha|\le j} \sup_{\bfx \in \Omega} |D^\alpha u(\bfx)|.
\]

\item The Sobolev embedding theorem \cite{Sobolev38} implies that $H^{k+m}(\Omega)$ can be continuously embedded 
into $C^{k}(\Omega)$ whenever $m\ge 1+\left[\frac{n}{2}\right]$; in particular, there exist constants $M_k$, depending only on $\Omega$, such that
\begin{equation}\label{Sobolev-embedding-estimate}
\|u\|_{k,\infty} \leq M_k \|u\|_{k + 1 + \left[\frac{n}{2}\right]}.
\end{equation}
\end{itemize}

\part{A Local Existence Theorem for Strongly Symmetric Positive Systems}

\section{Strong symmetric positivity}\label{ssp-sec}

Let $\Omega \subset \R^n$ be a bounded, open domain with piecewise smooth boundary $\partial \Omega$ and 
coordinates $\bfx = (x^1, \ldots, x^n)$. Let $\Phi:C^\infty(\bar{\Omega}, \R^s) \to C^\infty(\bar{\Omega}, \R^s)$ 
be a $C^\infty$, nonlinear first-order partial differential operator.  Explicitly, for $\bfu  \in C^\infty(\bar{\Omega}, \R^s)$, 
write
\[ 
\Phi(\bfu) = \mathbf{F}\left(\bfx, \bfu, \nabla\bfu \right), 
\]
where $\mathbf{F}(\bfx, \bfz, \mathbf{p} ) = (F^1(x^i, z^a, p^a_i), \ldots, F^s(x^i, z^a, p^a_i)) $ is a $C^\infty$, 
$\R^s$-valued function on $\bar{\Omega} \times \R^s \times \R^{ns}$.  
Given a function $\mathbf{f} \in C^\infty(\bar{\Omega}, \R^s)$, consider the PDE system:
\begin{equation}
\Phi(\bfu) = 
\mathbf{f} . \label{nonlinear-sys}
\end{equation}

The {\em linearization} of $\Phi$ at the function $\bfu_0 \in C^\infty(\bar{\Omega}, \R^s)$ is the linear first-order 
partial differential operator $\Phi'(\bfu_0): C^\infty(\bar{\Omega}, \R^s) \to C^\infty(\bar{\Omega}, \R^s)$ defined by
\begin{equation}
\Phi'(\bfu_0)\bfv = \frac{d}{dt}\bigg{\vert}_{t=0} \Phi(\bfu_0 + t\bfv) =
\sum_{i=1}^n A^i \partial_i \bfv + B \bfv , \label{linear-sys}
\end{equation}
where
$A^i, B \in C^\infty(\bar{\Omega}, \R^{s \times s})$ are given by
\begin{gather*}
 A^i(\bfx) =
\begin{bmatrix} \left(A^i(\bfx)\right)^a_b \end{bmatrix} =
 \begin{bmatrix} \displaystyle{\frac{\partial F^a}{\partial p^b_i}} (\bfx, \bfu_0(\bfx), \nabla \bfu_0(\bfx)) \end{bmatrix} , \\
 B(\bfx) =
 \begin{bmatrix} \left(B(\bfx)\right)^a_b \end{bmatrix} =
 \begin{bmatrix} \displaystyle{\frac{\partial F^a}{\partial z^b}}(\bfx, \bfu_0(\bfx), \nabla \bfu_0(\bfx)) \end{bmatrix}.
\end{gather*}
We will also consider the linear PDE system:
\begin{equation}\label{linear-sys-example}
\sum_{i=1}^n A^i \partial_i \bfv + B \bfv = \mathbf{h},
\end{equation}
where $\mathbf{h} \in C^\infty(\bar{\Omega}, \R^s)$.

\begin{definition}\label{ssp-def}
The linear partial differential operator \eqref{linear-sys} is called:
\begin{itemize}
\item {\em symmetric} if the matrices $A^1(\bfx), \ldots, A^n(\bfx)$ are symmetric for all $\bfx \in \bar{\Omega}$;
\item {\em symmetric positive} if it is symmetric and the quadratic form $Q_0(\bfx):\R^s \to \R$ defined by
\begin{equation}\label{define-Q0}
 Q_0(\bfx)(\xi) = \xi^{\sf T} \Big{(}B(\bfx) + B^{\sf T} (\bfx) - \sum_{i=1}^n \partial_i A^i(\bfx)\Big{)} \xi
\end{equation}
is positive definite for all $\bfx \in \bar{\Omega}$;
\item {\em strongly symmetric positive} if it is symmetric positive
and
the quadratic form $Q_1(\bfx):\R^{ns} \to \R$ defined by
\begin{equation}\label{define-Q1}
Q_1(\bfx)(\xi_1, \ldots, \xi_n) = \sum_{i,j=1}^n \xi_j^{\sf T}\left(\partial_j A^i(\bfx) + \partial_i A^j(\bfx)\right)\xi_i
\end{equation}
is positive definite for all $\bfx \in \bar{\Omega}$.
\end{itemize}
The nonlinear system \eqref{nonlinear-sys} is called {\em symmetric} (resp., {\em symmetric positive},
{\em strongly symmetric positive}) at $\bfu_0$ if the linearization \eqref{linear-sys} of $\Phi$ at $\bfu_0$ 
is symmetric (resp., symmetric positive, strongly symmetric positive).
\end{definition}

\begin{remark}
A few remarks are in order regarding Definition \ref{ssp-def}:

\begin{itemize}
\item The quadratic form $Q_1(\bfx)$ can be represented by the symmetric $ns \times ns$ matrix
\begin{equation}\label{define-Q1-matrix}
Q_1(\bfx) = \left[
\begin{MAT}[5pt]{c|c|c|c}
 2\, \partial_1 A^1(\bfx) & \partial_1 A^2(\bfx) + \partial_2 A^1(\bfx) & \cdots & \partial_1 A^n(\bfx) + \partial_n A^1(\bfx) \\-
 \partial_1 A^2(\bfx) + \partial_2 A^1(\bfx)  & 2\, \partial_2 A^2(\bfx) & \cdots & \partial_2 A^n(\bfx) + \partial_n A^2(\bfx) \\-
\vdots & \vdots & \vdots & \vdots   \\-
  \partial_1 A^n(\bfx) + \partial_n A^1(\bfx) & \partial_2 A^n(\bfx) + \partial_n A^2(\bfx) & \cdots & 2\, \partial_n A^n(\bfx) \\
\end{MAT}
   \right] .
\end{equation}
We will use the notation $(Q_1)_{ij}(\bfx)$ to denote the $(i,j)$th block of $Q_1(\bfx)$:
\[ (Q_1)_{ij}(\bfx) = \partial_i A^j(\bfx) + \partial_j A^i(\bfx). \]
\item The positivity of $Q_1(\bfx)$ is called the {\em Legendre condition} (\cite{Morrey08}, p.10).
Moser \cite{Moser1} established an existence theorem in the real analytic category
under the slightly weaker {\em Legendre-Hadamard condition} (\cite{Morrey08}, p.11), which requires only that
\begin{equation} \left((Q_1)_{ij}\right)_{ab}(\bfx) \xi^a \xi^b \eta^i \eta^j \geq \lambda |\xi|^2|\eta|^2 \label{Legendre-Hadamard}
\end{equation}
for all $\xi \in \R^s, \eta \in \R^n$, and some $\lambda > 0$.  However, in the $C^\infty$ category,
the stronger Legendre condition is necessary (cf. \cite{Serre06,YangMO}).
\end{itemize}
\end{remark}

\begin{definition}\label{P-convex-def}
Given a linear strongly symmetric positive first order partial differential operator $P = A^i \partial_i + B$ on a domain
$\Omega \subset \mathbb{R}^n$, the domain $\Omega$ is called {\em $P$-convex} if the {\em characteristic matrix}
$$
\beta(\bfx) = \sum_{i=1}^n \nu_i(\bfx) A^i(\bfx),
$$
is positive definite at each point $\bfx \in \partial\Omega$,
where $\nu(\bfx) = (\nu_1(\bfx), \ldots, \nu_n(\bfx))$ denotes the outer unit normal vector to $\partial \Omega$ at $\bfx \in \partial\Omega$,
\end{definition}

Tso \cite{Tso92} proved the following:

\begin{theoremvoid}[Theorem 5.1, \cite{Tso92}]
Suppose that $\Phi(\mathbf{0}) = \mathbf{0}$ and that the system \eqref{nonlinear-sys} is strongly symmetric 
positive at every $C^\infty$ function $\bfu$ in some $C^1$-neighborhood of the function $\bfu_0 = \mathbf{0}$ 
on a domain $\Omega \subset \R^n$ that is $P$-convex for the linearization $P$ of $\Phi$ at $\bfu_0 = \mathbf{0}$.
Then there exist an integer $\beta$ and a small constant $\epsilon>0$ such that, 
for any $\bff \in C^\infty(\bar{\Omega}, \R^s)$ with $\|\bff\|_\beta < \epsilon$, 
there exists a solution $\bfu \in C^\infty(\bar{\Omega}, \R^s)$ to the nonlinear system \eqref{nonlinear-sys} on $\bar{\Omega}$.
\end{theoremvoid}

\begin{remark}
Note that the condition that a PDE system is symmetric is not an open condition with respect to the coefficients. 
Since the Nash-Moser implicit function theorem requires solving the linearized equation not just at $\bfu_0$, 
but at all $\bfu$ near $\bfu_0$, it is necessary to assume that $\Phi'(\bfu)$ is symmetric for all $\bfu$ in 
some neighborhood of $\bfu_0$.  
The positivity conditions, however, are open conditions; hence it suffices to assume that they hold at $\bfu_0$.
\end{remark}

Moser \cite{Moser1} proved this theorem in the case where $\Phi$ and the function $\bff$ in equation \eqref{nonlinear-sys} 
are real analytic, under the weaker assumption of symmetric positivity together with 
the Legendre-Hadamard condition \eqref{Legendre-Hadamard}. 
Tso \cite{Tso92} stated this theorem assuming these same conditions; however, we believe that Tso's proof, 
which uses the G\"arding inequality for non-compactly-supported vector-valued functions on the domain $\Omega$, 
is correct only if the stronger Legendre condition holds. See \cite{Serre06} and the discussion at \cite{YangMO}.

\section{A local existence theorem for strongly symmetric positive systems}\label{state-theorem-sec}

The goal of Part 1 of this paper is to prove the following local version of Tso's theorem:

\begin{ourtheorem}\label{nonlinear-theorem}
Suppose that the linearization $\Phi'(\bfu)$ of $\Phi$ is symmetric for all $\bfu$ in some $C^1$-neighborhood
of $\bfu_0 \in C^\infty(\bar{\Omega}, \R^s)$, and that $\Phi'(\bfu_0)$ is strongly symmetric positive
at some point $\bfx_0 \in \Omega$.  Then there exist a neighborhood $\Omega_0 \subset \Omega$ of $\bfx_0$,
an integer $\beta$, and $\epsilon > 0$ such that, for any $\bff \in C^\infty(\Omega_0, \R^s)$
with $\|\Phi(\bfu_0) - \bff\|_\beta < \epsilon$,
there exists a solution $\bfu \in C^\infty(\Omega_0, \R^s)$ to the nonlinear system \eqref{nonlinear-sys} on $\Omega_0$.
\end{ourtheorem}

We wish to emphasize that Tso's theorem does not immediately imply the local existence result,
because strong symmetric positivity on a domain $\Omega$ does not necessarily guarantee
the existence of a $P$-convex neighborhood of $\bfx_0$.  In fact, as we show in the example below,
in general no such neighborhood exists.

\begin{example}\label{1d-example}
Consider the following ODE:
\begin{equation}\label{ODE-example}
(x - x_0) u' + bu = h(x)
\end{equation}
with $h\in C^\infty$.
It is straightforward
to verify that:
\begin{enumerate}
\item[\rm (i)]  \eqref{ODE-example} is strongly symmetric positive if $b > \tfrac{1}{2}${\rm ;}
\item[\rm (ii)] an interval $\Omega = (x_1, x_2)$ is $P$-convex if and only if $x_0 \in
(x_1,x_2)$, i.e., if and only if the regular singular point of this ODE lies in the
domain.
\end{enumerate}

Meanwhile, the general solution of \eqref{ODE-example} is
\[ u(x) = \frac{1}{(x - x_0)^b} \int_{x_0}^x (y - x_0)^{b-1} h(y)\, dy + \frac{C}{(x-x_0)^b}, \]
which is smooth at $x=x_0$ if and only if $C=0$.  Thus we see that:

\begin{itemize}
\item The $P$-convexity condition forces the uniqueness of a $C^\infty$ solution of \eqref{ODE-example} on $\Omega$,
{\em without} specifying any initial or boundary data for $u$.
\item If $\Omega$ is {\em not} $P$-convex---i.e., if $x_0 \notin \Omega$, then the ODE \eqref{ODE-example} has infinitely many solutions on $\Omega$.
In this case, $P$-convexity---and hence uniqueness of the solution---can be achieved by extending the domain to one that contains the singular point $x_0$.
\end{itemize}

In higher dimensions, a similar phenomenon occurs: Consider the strongly symmetric positive linear PDE system
\eqref{linear-sys-example}
on a domain $\Omega \subset \R^n$, and let $\bfx_1, \bfx_2 \in \partial\Omega$ be located on opposite sides of $\partial\Omega$,
with $\nu = \nu(\bfx_1) = -\nu(\bfx_2)$.
In order to have $\beta(\bfx_1), \beta(\bfx_2) > 0$, the matrix $\nu_i A^i(\bfx)$ must be positive definite at $\bfx_1$
and negative definite at $\bfx_2$.  Therefore, $P$-convexity requires that each of its eigenvalues must change sign somewhere in the interior of $\Omega$.
For $n \geq 2$, this does not necessarily imply that the system \eqref{linear-sys-example} has any singular points in $\Omega$,
but it is still true that any $C^\infty$ solution on $\Omega$ is unique.  Moser discussed
this in \cite{Moser1}, concluding that, ``The reason for this strange phenomenon is that usually the conditions [of the theorem]
imply the presence of a singularity and a solution which remains smooth at the singularity is unique."
\end{example}

Our proof of Theorem \ref{nonlinear-theorem} will proceed as follows: Without loss of generality, assume that $\bfu_0 = \mathbf{0}$ and $\Phi(\mathbf{0}) = \mathbf{0}$.
\begin{itemize}
\item In \S \ref{proof-1-step-1-sec}, we restrict the nonlinear system \eqref{nonlinear-sys} to a neighborhood 
$\Omega_0 \subset \Omega$ of $\bfx_0$ on which the quadratic forms $(Q_\bfu)_0(\bfx)$ and $(Q_\bfu)_1(\bfx)$ associated 
to any sufficiently small function $\bfu$ on $\Omega_0$ remain sufficiently 
close to $Q_0(\mathbf{0})$ and $Q_1(\mathbf{0})$, respectively.

\item In \S \ref{proof-1-step-2-sec}, we extend the linear PDE system \eqref{linear-sys-example} from the domain $\Omega_0$ to 
a strongly symmetric positive system on all of $\mathbb{R}^n$, where the coefficients satisfy $C^1$ bounds that will be needed later.

\item In \S \ref{boundary-conditions-subsec}, we show that, for sufficiently large $R>0$, the ball $B_R$ of 
radius $R$ is $P$-convex for the extended linear system.

\item In \S \ref{proof-1-step-3-sec}, we use the extended linear system on $B_R$ to prove the smooth tame estimates 
required to implement a Nash-Moser iteration scheme to solve the nonlinear system \eqref{nonlinear-sys} on $\Omega_0$.
\end{itemize}

Appendix \ref{classic-theorems-app} contains the precise statements of the Stein extension theorem \cite{Stein70} 
and the Nash-Moser implicit function theorem \cite{Sergeraert70} that will be used in the proof of Theorem \ref{nonlinear-theorem}.

\section{Proof of Theorem \ref{nonlinear-theorem}}\label{proof-sec}

\subsection{Restriction of the nonlinear system to an appropriate neighborhood of $\bfx_0$}\label{proof-1-step-1-sec}

Without loss of generality, assume that $\bfu_0 = \mathbf{0}$, $\Phi(\mathbf{0}) = \mathbf{0}$, and $\bfx_0 = \mathbf{0}$.
First, we show how to choose an appropriate neighborhood $\Omega_0$ on which to construct a solution for the system \eqref{nonlinear-sys}.

For ease of notation, set
\[ \bar{B} = B(\mathbf{0}), \qquad \bar{A}^i = A^i(\mathbf{0}), \qquad \bar{A}^i_j = \partial_j A^i(\mathbf{0}). \]
Using Taylor's theorem with remainder, we can write
\begin{equation}\label{Taylor-expansions}
B(\bfx) = \bar{B} + \hat{B}(\bfx), \qquad A^i(\bfx) = \bar{A}^i + \sum_{j=1}^n x^j \bar{A}^i_j + \hat{A}^i(\bfx),
\end{equation}
where $\hat{B}, \hat{A}^i \in C^\infty(\bar{\Omega}, \R^{s \times s})$ are such that $\hat{B}$ vanishes to 
order $1$ and $\hat{A}^i$ vanishes to order 2 at $\bfx=\mathbf{0}$.
The strong symmetric positivity hypothesis at $\bfx = \mathbf{0}$ is equivalent to the assumption 
that the quadratic forms $\bar{Q}_0:\R^s \to \R$ and $\bar{Q}_1:\R^{ns} \to \R$ defined by
\begin{equation}
\bar{Q}_0(\xi) = \xi^{\sf T}\Big{(}\bar{B} + \bar{B}^{\sf T} 
- \sum_{i=1}^n \bar{A}^i_i\Big{)} \xi, \qquad \bar{Q}_1(\xi_1, \ldots, \xi_n) 
= \sum_{i,j=1}^n \xi_j^{\sf T}\left(\bar{A}^i_j + \bar{A}^j_i\right)\xi_i
\end{equation}
are positive definite.

\begin{lemma}\label{choose-nbhd-lemma}
Suppose that $\Phi$ satisfies the hypotheses of Theorem {\rm \ref{nonlinear-theorem}} at $\bfx = \mathbf{0}$.  
Let $\lambda_0, \lambda_1 > 0$ denote the minimum eigenvalues of $\bar{Q}_0$ and $\bar{Q}_1$, respectively, 
and let $B_r \subset \R^n$ denote the ball of radius $r$ about $\bfx = \mathbf{0}$.  
Then, given real numbers $M_0, M_1 > 1$ and $\delta > 0$, there exist real numbers $r, \rho>0$ 
and an integer $\alpha>0$ such that $B_r \subset \Omega$ and, 
for any $\bfu \in C^\infty(B_r, \R^s)$ with $\|\bfu\|_\alpha < \rho$, 
the matrix-valued functions $B_\bfu, A^i_\bfu \in C^\infty(B_r, \R^{s \times s})$ associated 
to the linearization of $\Phi$ at $\bfu$ may be written as
\[
B_\bfu(\bfx) = \bar{B}_\bfu + \hat{B}_\bfu(\bfx), 
\qquad A^i_\bfu(\bfx) = \bar{A}_\bfu^i + \sum_{j=1}^n x^j (\bar{A}_\bfu)^i_j + \hat{A}_\bfu^i(\bfx),
\]
where
\begin{equation}\label{perturbed-bounds}
\begin{gathered}
 |\bar{B}_\bfu - \bar{B}|_\infty < \frac{\delta}{2}, 
 \qquad |(\bar{A}_\bfu)^i_j - \bar{A}^i_j|_\infty < \frac{\delta}{2}, \qquad  |\bar{A}_\bfu^i - \bar{A}^i|_\infty < \delta, \\
 \|\hat{B}_{\mathbf{u}}\|_{0,\infty} < \frac{\delta}{2M_0}, \qquad \|\hat{A}_{\mathbf{u}}\|_{1,\infty} < \frac{\delta}{2M_1}, \qquad
 \|\hat{A}_{\mathbf{u}}\|_{0,\infty} < \frac{\delta}{M_0}.
\end{gathered}
\end{equation}
\end{lemma}

For convenience, we will refer to any function $\bfu \in C^\infty(B_r, \R^s)$ with $\|\bfu\|_\alpha < \rho$ as ``admissible."

\begin{proof}
Choose $r > 0$ so that the restrictions of
$\hat{B}$ and $\hat{A}^i$
to the ball $B_r$ of radius $r$ satisfy
\begin{equation}  \label{bounds-on-hot}
\|\hat{B}\|_{0,\infty} < \frac{\delta}{4M_0}, \qquad
\|\hat{A}^i\|_{1,\infty} < \frac{\delta}{4M_1}, \qquad
\|\hat{A}^i\|_{0,\infty} < \frac{\delta}{2M_0}.
\end{equation}
Then the Sobolev embedding estimate \eqref{Sobolev-embedding-estimate} and the smallness of the Taylor remainder 
terms for small $\rho$ imply that we may choose $\rho$ and $\alpha$ so that equations \eqref{perturbed-bounds} hold.
Indeed, we may choose any $\alpha \geq 3 + [\frac{n}{2}]$ and then choose $\rho>0$ accordingly.
\end{proof}

In \S \ref{proof-1-step-2-sec}, we will show how to choose the constants $\delta, M_0$, and $M_1$ so that the restriction 
of the system \eqref{nonlinear-sys} to the domain $\Omega_0 = B_r$ has the property that its linearization 
at any admissible $\bfu \in C^\infty(B_r, \R^s)$ may be extended to a strongly symmetric positive system on all of $\R^n$.

\subsection{Extension of the linearized system to $\R^n$}\label{proof-1-step-2-sec}

We will use Stein's extension operator (cf. Theorem \ref{Stein-theorem}) to extend the coefficient matrices in the 
linearized system \eqref{linear-sys-example} from $B_r$ to all of $\R^n$.  
First we need the following lemma, which states that the bounding constants in this construction are independent of $r$:

\begin{lemma}\label{dilation-lemma}
There exist constants $M_{k,p}, \ 1 \leq p \leq \infty, \ 0 \leq k < \infty$, 
and extension operators $\scE_r: L^1(B_r) \to L^1(\R^n)$ for all $r>0$ such that, for all $f \in W^{k,p}(B_r)$,
\[ 
\| \scE_r f \|_{k,p} \leq M_{k,p} \| f \|_{k,p}. 
\]
\end{lemma}

\begin{proof}
Theorem \ref{Stein-theorem} guarantees the existence of such constants and an extension operator for $r=1$; 
then a straightforward rescaling of the operator and a standard rescaling argument shows that these constants are independent of $r$.
\end{proof}

Now, set
\[ 
M_0 = M_{0,\infty}, \qquad M_1 = M_{1,\infty}, 
\]
where $M_{0,\infty}$ and $M_{1,\infty}$ are as in Lemma \ref{dilation-lemma}.
Choose $\delta > 0$ such that, for any matrices $\bar{A}^i_j{}'$ and $\bar{B}'$ with
\[  
|\bar{B}' - \bar{B}|_\infty < \delta,  \qquad |\bar{A}^i_j{}' - \bar{A}^i_j|_\infty < \delta, \qquad  1 \leq i,j \leq n,  
\]
the quadratic forms $\bar{Q'}_0:\R^s \to \R$ and $\bar{Q}'_1:\R^{ns} \to \R$
defined by
\[
\bar{Q}'_0(\xi) = \xi^{\sf T}\Big{(}\bar{B}'+ (\bar{B'})^{\sf T} - \sum_{i=1}^n \bar{A}^i_i{}'\Big{)} \xi, 
\qquad \bar{Q}'_1(\xi_1, \ldots, \xi_n) = \sum_{i,j=1}^n \xi_j^{\sf T}\left(\bar{A}^i_j{}' + \bar{A}^j_i{}'\right)\xi_i 
\]
are positive definite with minimum eigenvalues greater than or equal to $\tfrac{1}{2}\lambda_0$ and $\tfrac{1}{2} \lambda_1$, respectively.  
Then take $r>0$ as given by Lemma \ref{choose-nbhd-lemma}, and set $\Omega_0 = B_r$.
Henceforth, we will restrict the systems \eqref{nonlinear-sys} and \eqref{linear-sys-example} and all relevant quantities to $B_r$.

Next, we construct an extension of the linearized system \eqref{linear-sys-example} on $B_r$ to all of $\R^n$ 
in such a way that the coefficients of the extended system are bounded in $W^{k,p}(\R^n)$ with respect 
to the $W^{k,p}(B_r)$ norms of the coefficients of the original system on $B_r$.
After replacing the functions $\hat{B}$, $\hat{A}^i$, and $\mathbf{h}$ by their restrictions to $B_r$, 
define $C^\infty$ functions $\tilde{B}, \tilde{A}^i$, and $\tilde{\mathbf{h}}$ on $\R^n$ by
\begin{equation}\label{defined-extended-functions}
\begin{aligned}
\tilde{B}(\bfx) & = \bar{B} + (\scE_r \hat{B})(\bfx), \\
\tilde{A}^i(\bfx) & = \bar{A}^i + \sum_{j=1}^n x^j \bar{A}^i_j + (\scE_r \hat{A}^i)(\bfx), \\
\tilde{\mathbf{h}}(\bfx) & = (\scE_r\mathbf{h})(\bfx).
\end{aligned}
\end{equation}
Similarly, for any admissible $\bfu \in C^\infty(B_r, \R^s)$, let $\tilde{A}^i_\bfu$ and $\tilde{B}_\bfu$ denote 
the analogous extensions of the functions $A^i_\bfu$ and $B_\bfu$ corresponding to the linearization of $\Phi$ at $\bfu$.
Then we have the extended linear systems
\begin{equation}
\sum_{i=1}^n \tilde{A}_\bfu^i \partial_i \tilde{\bfv} + \tilde{B}_\bfu \tilde{\bfv} = \tilde{\bfh} \label{extended-linear-sys}
\end{equation}
on $\R^n$.

\begin{proposition}\label{extended-ssp-prop}
For any admissible $\bfu \in C^\infty(B_r, \R^s)$,
the extended system \eqref{extended-linear-sys} is strongly symmetric positive on $\R^n$.  
Moreover, for any $\bfx \in \R^n$, the associated quadratic forms $(\tilde{Q}_\bfu)_0(\bfx):\R^s \to \R$ and $(\tilde{Q}_\bfu)_1(\bfx):\R^{ns} \to \R$ defined by
\begin{equation}\label{extended-quad-forms}
\begin{gathered}
 (\tilde{Q}_\bfu)_0(\bfx)(\xi) = \xi^{\sf T} \Big{(}\tilde{B}_\bfu(\bfx) + \tilde{B}_\bfu^{\sf T}(\bfx) - \sum_{i=1}^n \partial_i \tilde{A}_\bfu^i(\bfx)\Big{)} \xi, \\
(\tilde{Q}_\bfu)_1(\bfx)(\xi_1, \ldots, \xi_n) = \sum_{i,j=1}^n \xi_j^{\sf T}\left(\partial_j \tilde{A}_\bfu^i(\bfx) + \partial_i \tilde{A}_\bfu^j(\bfx)\right)\xi_i
\end{gathered}
\end{equation}
have minimum eigenvalues greater than or equal to $\tfrac{1}{2}\lambda_0$ and $\tfrac{1}{2} \lambda_1$, respectively.
\end{proposition}

\begin{proof}
By construction, the functions $\scE_r \hat{B}_\bfu$ and $\scE_r \hat{A}_\bfu^i$ satisfy
\begin{equation}  \label{bounds-on-extended-hot}
\|\scE_r \hat{B}_\bfu \|_{0,\infty} < \frac{\delta}{2}, \qquad
\|\scE_r \hat{A}_\bfu^i \|_{1,\infty} < \frac{\delta}{2} , \qquad
\|\scE_r \hat{A}_\bfu^i \|_{0,\infty} < \delta.
\end{equation}
The first and second inequalities in \eqref{bounds-on-extended-hot} imply that, for all $\bfx \in \R^n$, we have
\begin{gather*}
 |\tilde{B}_\bfu(\bfx) - \bar{B}|_\infty \leq |\bar{B}_\bfu - \bar{B}|_\infty + |\scE_r\hat{B}_\bfu(\bfx)|_\infty < \delta,  \\
 |\partial_j\tilde{A}^i_\bfu(\bfx) - \bar{A}^i_j|_\infty \leq |(\bar{A}_\bfu)^i_j - \bar{A}^i_j|_\infty + |\partial_j (\scE_r\hat{A}^i_\bfu)(\bfx)|_\infty < \delta,
\end{gather*}
and the result follows immediately.
\end{proof}

\subsection{Boundary conditions on $B_R$ for large $R$}\label{boundary-conditions-subsec}

Next, we show that, for $R$ sufficiently large, $B_R$ is $P$-convex for the extended linear system \eqref{extended-linear-sys}.

\begin{proposition}\label{good-bcs-prop}
Let $R>0$.  For $\bfx \in \partial B_R$, let $\nu(\bfx) = (\nu_1(\bfx), \ldots, \nu_n(\bfx))$ denote 
the outward-pointing unit normal vector to $\partial B_R$ at $\bfx$.  Then, for $R$ sufficiently large, the characteristic matrix
\begin{equation}\label{define-beta}
 \beta(\bfx) = \sum_{i=1}^n \nu_i(\bfx) \tilde{A}_\bfu^i(\bfx)
\end{equation}
is positive definite for all admissible $\bfu \in C^\infty(B_r, \R^s)$ and  $\bfx \in \partial B_R$.
\end{proposition}

\begin{proof}
The normal vector to the sphere $\partial B_R$ is given by
\[ \nu(\bfx) =  \frac{1}{R} \bfx. \]
Thus, we have
\begin{equation}\label{beta-formula}
\begin{aligned}
\beta(\bfx) & = \frac{1}{R} \sum_{i=1}^n x^i \tilde{A}_\bfu^i(\bfx) \\
&= \frac{1}{R} \Big{(} \sum_{i=1}^n x^i \bar{A}_\bfu^i + \sum_{i,j=1}^n x^i x^j (\bar{A}_\bfu)^i_j 
+ \sum_{i=1}^nx^i (\scE_r \hat{A}_\bfu^i)(\bfx) \Big{)}.
\end{aligned}
\end{equation}
The first and third terms in equation \eqref{beta-formula} are bounded:
\begin{equation}\label{some-bounds}
\begin{aligned}
&\Big{\vert}\frac{1}{R} \sum_{i=1}^n x^i \bar{A}_\bfu^i \Big{\vert}_\infty  
  < n \max \left\{ |\bar{A}_\bfu^1|_\infty, \ldots, |\bar{A}_\bfu^n|_\infty \right\}, \\
&\Big{\vert}\frac{1}{R}  \sum_{i=1}^n x^i (\scE_r \hat{A}_\bfu^i)(\bfx)  \Big{\vert}_\infty  < n \delta,
\end{aligned}
\end{equation}
where the second equation in \eqref{some-bounds} follows from the third inequality in \eqref{bounds-on-extended-hot}.  
Meanwhile, we claim that the second term in equation \eqref{beta-formula} has a minimum eigenvalue greater 
than or equal to $\tfrac{1}{4}R\lambda_1$.  This can be seen as follows:
Consider the corresponding quadratic form $(\tilde{Q}_\bfu)^\dagger_1(\bfx): \R^s \to \R$ given by
\[ 
(\tilde{Q}_\bfu)^\dagger_1(\bfx)(\xi) = \frac{1}{R}\sum_{i,j=1}^n \xi^{\sf T} (x^i x^j (\bar{A}_\bfu)^i_j) \xi.
\]
Then, by Proposition \ref{extended-ssp-prop}, we have
\begin{align*}
 |(\tilde{Q}_\bfu)^\dagger_1(\bfx)(\xi)| & = \frac{1}{2R}|(\tilde{Q}_\bfu)_1(\bfx) (x^1 \xi, \ldots, x^n \xi)| \\
 & \geq \frac{1}{4R}\lambda_1 \left( |x^1 \xi|^2 + \ldots + |x^n \xi|^2 \right) \\
 & = \frac{1}{4R}\lambda_1  ((x^1)^2 + \ldots + (x^n)^2) |\xi|^2  \\
 & = \frac{1}{4R}\lambda_1  R^2 |\xi|^2  \\
 & = \tfrac{1}{4} R \lambda_1 |\xi|^2.
\end{align*}
Therefore, the minimum eigenvalue of $(\tilde{Q}_\bfu)^\dagger_1(\bfx)$ is greater than or equal to $\tfrac{1}{4}R\lambda_1$.  
Together with the inequalities in \eqref{some-bounds}, 
this implies that, for $R$ sufficiently large, $\beta(\bfx)$ is positive definite for all $\bfx \in \partial B_R$.
\end{proof}

\subsection{Application of the Nash-Moser iteration scheme}\label{proof-1-step-3-sec}

The final step in the proof of Theorem \ref{nonlinear-theorem} is to apply the Nash-Moser implicit function theorem (cf. Theorem \ref{Nash-Moser-theorem}).

\begin{notation}
We will adopt the following conventions:
\begin{itemize}
\item Functions without tildes are taken to be defined on $B_r$, and $\|\bfv\|_k$ will denote the $H^k$-norm of $\bfv \in H^k(B_r)$.
\item Functions with tildes are taken to be defined on $B_R$, and $\|\tilde{\bfv}\|_k$ will denote the $H^k$-norm of $\tilde{\bfv} \in H^k(B_R)$.
\end{itemize}
\end{notation}
Let  $E_k = H^{k+1}(B_r, \R^s)$ and $F_k = H^k(B_r, \R^s)$, with the usual $H^k$-norms; 
then we have $E_\infty = F_\infty = C^\infty(B_r, \R^s)$.  
Let $D_0 \subset E_0$ denote the ball of radius $\rho > 0$ centered at $\bfu_0$.

Smoothing operators $S(t):E_0 \to E_\infty$ may be constructed as follows  (see., e.g., \cite{BGY83} or \cite{Schwartz69}).
First, choose a compactly supported function $\chi \in C^\infty_0(\R^n)$ with $\chi \geq 0$ 
and $\int_{\R^n} \chi(\bfx)\, d\bfx = 1$.  For $t > 0$, define
\[ \chi_t(\bfx) = t^n \chi(t\bfx), \]
and define $\hat{S}_t:L^2(\R^n, \R^s) \to C^\infty(\R^n, \R^s)$ by
\[ 
(\hat{S}_t\hat{\bfu})(\bfx) = \int_{\R^n} \chi_t(\bfx - \bfy) \bfu(\bfy)\, d\bfy. 
\]
Then, define $S_t: E_0 \to E_\infty$ by composing $\hat{S}_t$ with the Stein extension 
operator $\scE_r: L^1(B_r) \to L^1(\R^n)$: For $\bfu \in E_0 = H^1(B_r, \R^s)$, define
\[ 
(S_t\bfu) = (\hat{S}_t \scE_r\bfu) \vert_{B_r}. 
\]
It is straightforward to show that the operators $S_t$ satisfy the required 
inequalities; see \cite{Schwartz69} for details.

The fact that $\Phi$ is $C^2$ follows from the fact that $\mathbf{F}$ is $C^\infty$, 
and the bounds \eqref{Nash-Moser-Phi-bounds} for any $\alpha \geq 0$ follow from 
the Gagliardo-Nirenberg and Sobolev inequalities (see, e.g., \cite{Evans10}).
To complete the proof, it suffices to show that there exists an integer $\alpha \geq 0$ such that, 
for any integer $m \geq \alpha+1$, given any $\bfu \in D_m$, the extended linear system
\eqref{extended-linear-sys}
on $B_R$ corresponding to the linearization of \eqref{nonlinear-sys} at $\bfu$ has 
a unique solution $\tilde{\bfv} \in H^{m - \alpha}(B_R)$ for any $\tilde{\bfh} \in H^m(B_R)$, 
and that the restriction $\bfv = \tilde{\bfv}\vert_{B_r}$ satisfies 
the smooth tame estimates \eqref{Nash-Moser-estimates}.

First, because the extended system \eqref{extended-linear-sys} corresponding to a given 
admissible $\bfu \in H^m(B_r)$ is symmetric positive with coefficient matrices (omitting the subscript $\bfu$ to avoid notational clutter) 
$\tilde{A}^1, \ldots, \tilde{A}^n, \tilde{B}\in H^{m-1}(B_R)$ and $B_R$ is $P$-convex for \eqref{extended-linear-sys}, 
Friedrichs's theory of symmetric positive systems \cite{Friedrichs58} guarantees the existence of 
a unique solution $\tilde{\bfv} \in L^2(B_R)$.  
Moreover, we can obtain an explicit $L^2$ bound for $\tilde{\bfv}$, and hence for $\bfv$, as follows.   
Multiply the matrix equation \eqref{extended-linear-sys} by $\tilde{\bfv}^{\sf T}$ to obtain the scalar equation
\begin{equation}
\sum_{i=1}^n \tilde{\bfv}^{\sf T} \tilde{A}^i \, \partial_i \tilde{\bfv} + \tilde{\bfv}^{\sf T} \tilde{B}\, 
\tilde{\bfv} = \tilde{\bfv}^{\sf T} \tilde{\bfh}. 
\label{estimate-0-step-1}
\end{equation}
Then, because $\tilde{A}^i$ is symmetric and
\[ 
\sum_{i=1}^n \partial_i \left( \tilde{\bfv}^{\sf T} \tilde{A}^i \tilde{\bfv}\right) 
= \sum_{i=1}^n \left( 2 \tilde{\bfv}^{\sf T} \tilde{A}^i\, (\partial_i \tilde{\bfv}) 
+ \tilde{\bfv}^{\sf T} (\partial_i \tilde{A}^i) \tilde{\bfv} \right), 
\]
we can write equation \eqref{estimate-0-step-1} as
\begin{equation}
-\tfrac{1}{2} \sum_{i=1}^n \tilde{\bfv}^{\sf T} (\partial_i \tilde{A}^i) \tilde{\bfv} 
+ \tilde{\bfv}^{\sf T} \tilde{B}\, \tilde{\bfv} = \tilde{\bfv}^{\sf T} \tilde{\bfh} 
- \tfrac{1}{2} \sum_{i=1}^n \partial_i \left( \tilde{\bfv}^{\sf T} \tilde{A}^i \tilde{\bfv}\right).
\end{equation}
Multiply by $2$ and use the fact that 
$\tilde{\bfv}^{\sf T} \tilde{B}\, \tilde{\bfv} = \tfrac{1}{2} \tilde{\bfv}^{\sf T} (\tilde{B} + \tilde{B}^{\sf T})\, \tilde{\bfv}$ to obtain
\begin{equation}
\tilde{\bfv}^{\sf T} \Big{(}\tilde{B} + \tilde{B}^{\sf T} - \sum_{i=1}^n \partial_i \tilde{A}^i\Big{)} \tilde{\bfv} 
= 2 \tilde{\bfv}^{\sf T} \tilde{\bfh} - \sum_{i=1}^n \partial_i \left( \tilde{\bfv}^{\sf T} \tilde{A}^i \tilde{\bfv}\right).  
\label{estimate-0-step-2}
\end{equation}
By Proposition \ref{extended-ssp-prop}, it follows that
\begin{align*}
 \tfrac{1}{2}\lambda_0 |\tilde{\bfv}|^2 & \leq 2 \tilde{\bfv}^{\sf T} \tilde{\bfh} 
 - \sum_{i=1}^n \partial_i \left( \tilde{\bfv}^{\sf T} \tilde{A}^i \tilde{\bfv}\right) \\
 & \le \frac{\lambda_0}{4}|\tilde{\bfv}|^2 +\frac{4}{\lambda_0}|\tilde{\bfh}|^2 
 -\sum_{i=1}^n\partial_i\left(\tilde{\mathbf{v}}^{\sf T}\tilde{A}^i\tilde{\mathbf{v}}\right).
\end{align*}
Integrate over $B_R$, apply Stokes' theorem, and use the fact that $\beta$ is positive definite on $\partial B_R$ to obtain
\begin{equation}
\|\tilde{\mathbf{v}}\|_0^2
\le C_0(\lambda_0)^2\|\tilde{\mathbf{h}}\|^2_0
-\frac{4}{\lambda_0}\int_{\partial B_R} \tilde{\mathbf{v}}^{\sf T} \beta\tilde{\mathbf{v}}\,dS
\le  C_0(\lambda_0)^2\|\tilde{\mathbf{h}}\|^2_0,
\label{v-tilde-estimate-0}
\end{equation}
where $C(\lambda_0)>0$ is a universal constant depending on $\lambda_0$.
Therefore, the restriction $\bfv$ of $\tilde{\bfv}$ to $B_r$ satisfies
\begin{equation}\label{v-order-0}
\|\bfv\|_0 \leq \|\tilde{\bfv}\|_0 \leq C_0(\lambda_0)  \|\tilde{\bfh}\|_0 \leq C_0(\lambda_0) M_{0,2}  \|\bfh\|_0,
\end{equation}
where $M_{0,2}$ is as in Lemma \ref{dilation-lemma}.

The bounds on the derivatives of $\bfv$ may be computed similarly by differentiation, 
and then the existence of these derivatives follows from standard results in analysis. First, differentiate the system \eqref{extended-linear-sys} with respect to $x^j$ to obtain
\begin{equation}
\sum_{i=1}^n \left(\tilde{A}^i\, \partial^2_{ij} \tilde{\bfv} 
+ (\partial_j \tilde{A}^i) \partial_i \tilde{\bfv} \right) + \tilde{B} \, \partial_j \tilde{\bfv} 
+ (\partial_j \tilde{B}) \tilde{\bfv} = \partial_j \tilde{\mathbf{h}}.  
\label{estimate-1-step-1}
\end{equation}
Multiply the matrix equation \eqref{estimate-1-step-1} by $\partial_j \tilde{\bfv}^{\sf T}$ to obtain
\begin{equation}
\sum_{i=1}^n \left( (\partial_j \tilde{\bfv}^{\sf T}) \tilde{A}^i\, \partial^2_{ij} \tilde{\bfv} 
+ (\partial_j \tilde{\bfv}^{\sf T})(\partial_j \tilde{A}^i) \partial_i \tilde{\bfv} \right) 
+ (\partial_j \tilde{\bfv}^{\sf T}) \tilde{B} \, \partial_j \tilde{\bfv} 
+ (\partial_j \tilde{\bfv}^{\sf T}) (\partial_j \tilde{B}) \tilde{\bfv} 
= (\partial_j \tilde{\bfv}^{\sf T}) \partial_j \tilde{\mathbf{h}}.  
\label{estimate-1-step-2}
\end{equation}
By an argument similar to that above, we can write equation \eqref{estimate-1-step-2} as
\begin{multline}
\qquad \partial_j \tilde{\bfv}^{\sf T} \Big{(}\tilde{B} + \tilde{B}^{\sf T} 
- \sum_{i=1}^n \partial_i \tilde{A}^i \Big{)} \partial_j \tilde{\bfv} 
+ 2\sum_{i=1}^n (\partial_j \tilde{\bfv}^{\sf T}) (\partial_j \tilde{A}^i) \partial_i \tilde{\bfv}  \\
= 2 \partial_j \tilde{\bfv}^{\sf T} \left(\partial_j \tilde{\mathbf{h}} - (\partial_j \tilde{B}) \tilde{\bfv}\right) 
- \sum_{i=1}^n \partial_i\left( \partial_j \tilde{\bfv}^{\sf T} \tilde{A}^i \partial_j \tilde{\bfv}\right). \qquad \label{estimate-1-step-3}
\end{multline}
Now sum equation \eqref{estimate-1-step-3} from $j=1$ to $n$, and note that the second term can be written as
\[ 
2 \sum_{i,j=1}^n (\partial_j \tilde{\bfv}^{\sf T}) (\partial_j \tilde{A}^i) \partial_i \tilde{\bfv} 
= \sum_{i,j=1}^n (\partial_j \tilde{\bfv}^{\sf T}) (\partial_j \tilde{A}^i + \partial_i \tilde{A}^j) (\partial_i \tilde{\bfv}). 
\]
Thus the summed equation can be written as
\begin{multline}
 \sum_{j=1}^n \partial_j \tilde{\bfv}^{\sf T} \Big{(}\tilde{B} + \tilde{B}^{\sf T} 
 - \sum_{i=1}^n \partial_i \tilde{A}^i \Big{)} \partial_j \tilde{\bfv} 
 + \sum_{i,j=1}^n (\partial_j \tilde{\bfv}^{\sf T}) (\partial_j \tilde{A}^i + \partial_i \tilde{A}^j) (\partial_i \tilde{\bfv})  \\
= \sum_{j=1}^n \left(2 \partial_j \tilde{\bfv}^{\sf T} \left(\partial_j \tilde{\mathbf{h}} 
- (\partial_j \tilde{B}) \tilde{\bfv}\right) - \sum_{i=1}^n \partial_i\left( \partial_j \tilde{\bfv}^{\sf T} \tilde{A}^i \partial_j \tilde{\bfv} \right)\right), 
\qquad \label{estimate-1-step-4a}
\end{multline}
or, in other words,
\begin{equation}
\sum_{j=1}^n  \tilde{Q}_0 (\partial_j \tilde{\bfv}) + \tilde{Q}_1( \partial_1 \tilde{\bfv}, \ldots, \partial_n \tilde{\bfv})  \\
 = \sum_{j=1}^n \left(2 \partial_j \tilde{\bfv}^{\sf T} \left(\partial_j \tilde{\mathbf{h}} 
 - (\partial_j \tilde{B}) \tilde{\bfv}\right) 
 - \sum_{i=1}^n \partial_i\left( \partial_j \tilde{\bfv}^{\sf T} \tilde{A}^i \partial_j \tilde{\bfv} \right)\right) .
\label{estimate-1-step-4}
\end{equation}
By Proposition \ref{extended-ssp-prop}, it follows that
\begin{align*}
\tfrac{1}{2} (\lambda_0 + \lambda_1) \sum_{j=1}^n |\partial_j \tilde{\bfv}|^2  
& \leq \sum_{j=1}^n \left(2 \partial_j \tilde{\bfv}^{\sf T} \left(\partial_j \tilde{\mathbf{h}} - (\partial_j \tilde{B}) \tilde{\bfv}\right) 
- \sum_{i=1}^n \partial_i\left( \partial_j \tilde{\bfv}^{\sf T} \tilde{A}^i \partial_j \tilde{\bfv} \right)\right) \\
& \le \tfrac{1}{4} (\lambda_0+\lambda_1) \sum_{j=1}^n |\partial_j \tilde{\bfv}|^2
+\frac{4}{(\lambda_0+\lambda_1)}\sum_{j=1}^n
\left( |\partial_j \tilde{\bfh}|^2
+|\partial_j \tilde{B}|_{0,\infty}^2|\tilde{\mathbf{v}}|^2\right) \\
& \qquad
- \sum_{i=1}^n\partial_i\left(\partial_j\tilde{\mathbf{v}}^\top\tilde{A}^i\partial_j\tilde{\mathbf{v}}\right).
\end{align*}
Integrate over $B_R$, apply Stokes' theorem, and  use the fact that $\beta$ is positive definite on $\partial B_R$ again to obtain
\begin{equation}\label{estimate-1-some-step}
\begin{aligned}
\|\tilde{\mathbf{v}}\|_1^2
&\le  C_1(\lambda_0, \lambda_1)^2 \left(\|\tilde{\mathbf{h}}\|^2_1
+\|\tilde{\mathbf{v}}\|^2_0\|\tilde{B}\|^2_{1,\infty}\right)
-\frac{4}{(\lambda_0+\lambda_1)}\sum_{j=1}^n\int_{\partial B_R} (\partial_j\tilde{\mathbf{v}}^{\sf T}\beta\partial_j \tilde{\mathbf{v}})\, dS\\
&\le   C_1(\lambda_0, \lambda_1)^2\left(\|\tilde{\mathbf{h}}\|^2_1
+\|\tilde{\mathbf{v}}\|^2_0\|\tilde{B}\|^2_{1,\infty}\right) ,
\end{aligned}
\end{equation}
where $C_1(\lambda_0,\lambda_1)>0$ is a universal constant depending on $\lambda_0$ and $\lambda_1$.

By the Sobolev embedding estimate \eqref{Sobolev-embedding-estimate}, we have
\[ \|\tilde{B}\|_{1,\infty} \leq K \|\tilde{B}\|_{2+[\frac{n}{2}]} \]
for some constant $K$; thus we can write the inequality \eqref{estimate-1-some-step} as
\[ \|\tilde{\mathbf{v}}\|_1^2 \leq C_1^2\left(\|\tilde{\mathbf{h}}\|^2_1
+\|\tilde{\mathbf{v}}\|^2_0\|\tilde{B}\|^2_{2+[\frac{n}{2}]}\right), \]
and hence
\[ \|\tilde{\mathbf{v}}\|_1 \leq C_1\left(\|\tilde{\mathbf{h}}\|_1
+\|\tilde{\mathbf{v}}\|_0\|\tilde{B}\|_{2+[\frac{n}{2}]}\right). \]
Therefore, the restriction $\bfv$ of $\tilde{\bfv}$ to $B_r$ satisfies
\begin{equation}\label{v-order-1}
\begin{aligned}
\|\bfv\|_1 \leq \|\tilde{\bfv}\|_1 & \leq C_1 \left( \|\tilde{\mathbf{h}}\|_1
+\|\tilde{\mathbf{v}}\|_0\|\tilde{B}\|_{2+[\frac{n}{2}]}\right) \\
& \leq C_1 M_{1,2} \left( \|\mathbf{h}\|_1
+\|\mathbf{v}\|_0\|B\|_{2+[\frac{n}{2}]}\right) \\
& \leq C'_1 \left( \|\mathbf{h}\|_1
+\|\mathbf{h}\|_0\|B\|_{2+[\frac{n}{2}]}\right) \\
& \leq C''_1 \left( \|\mathbf{h}\|_1
+\|\mathbf{h}\|_0\|\bfu\|_{3+[\frac{n}{2}]}\right),
\end{aligned}
\end{equation}
where the last inequality follows from the fact that $B$ is a $C^\infty$ function of $\bfu$ and its first derivatives.

Successive differentiations of the system \eqref{extended-linear-sys} produce similar results.  
To obtain an estimate for $\|\bfv\|_k$, differentiate the system \eqref{extended-linear-sys} $k$ times, 
with respect to $x^{j_1}, \ldots, x^{j_k}$.  This yields an equation of the form
\begin{multline}
\sum_{i=1}^n \left( \tilde{A}^i\, \partial^{k+1}_{ij_1\ldots j_k} \tilde{\bfv} 
+ \sum_{q=1}^k (\partial_{j_q} \tilde{A}^i) \partial^k_{i j_1 \ldots \hat{j}_q \ldots j_k} \tilde{\bfv} \right) 
+ \tilde{B}\, \partial^k_{j_1\ldots j_k} \tilde{\bfv}  \\
= \partial^k_{j_1\ldots j_k} \tilde{\mathbf{h}} - ( \partial^k_{j_1\ldots j_k} \tilde{B}) \tilde{\bfv} 
- \sum_{m=1}^{k-1} \left(\sum_{i=1}^n D^{k+1-m}  \tilde{A}^i + D^{k-m} \tilde{B}  \right)  ( D^{m}  \tilde{\bfv} ),
  \label{estimate-k-step-1}
\end{multline}
where, on the right-hand side, $D^{m}$ indicates an appropriate differential operator of order $m$.  
Multiply the matrix equation \eqref{estimate-k-step-1} by $2 \partial^k_{j_1\ldots j_k} \tilde{\bfv}^{\sf T}$, 
rewrite the first term and rearrange as in the previous cases, so that the left-hand side of equation \eqref{estimate-k-step-1} becomes
\begin{equation}
\partial^k_{j_1\ldots j_k} \tilde{\bfv}^{\sf T} (\tilde{B} + \tilde{B}^{\sf T} - \sum_{i=1}^n \partial_i \tilde{A}^i) \partial^k_{j_1\ldots j_k} \tilde{\bfv} +
2 (\partial^k_{j_1\ldots j_k} \tilde{\bfv}^{\sf T}) \sum_{i=1}^n \sum_{q=1}^k (\partial_{j_q} \tilde{A}^i) \partial^k_{i j_1 \ldots \hat{j}_q \ldots j_k} \tilde{\bfv}.
\label{estimate-k-step-3}
\end{equation}
Now sum over $j_1, \ldots, j_k$, and note that the second term in \eqref{estimate-k-step-3} can be rearranged as follows 
by using the commutativity of mixed partial derivatives and relabeling as appropriate:
\begin{align*}
& \ 2 \sum_{i, j_1, \ldots, j_k = 1}^n \sum_{q=1}^k  (\partial^k_{j_1\ldots j_k} 
\tilde{\bfv}^{\sf T})(\partial_{j_q} \tilde{A}^i) \partial^k_{i j_1 \ldots \hat{j}_q \ldots j_k} \tilde{\bfv} \\
= & \
2 \sum_{i, j_1, \ldots, j_k = 1}^n \sum_{q=1}^k (\partial^k_{j_q j_1 \ldots \hat{j}_q \ldots j_k } 
 \tilde{\bfv}^{\sf T})(\partial_{j_q} \tilde{A}^i) \partial^k_{i j_1 \ldots \hat{j}_q \ldots j_k} \tilde{\bfv} \\
= & \sum_{i, j_1, \ldots, j_k = 1}^n \sum_{q=1}^k (\partial^k_{j_q j_1 \ldots \hat{j}_q \ldots j_k } 
 \tilde{\bfv}^{\sf T})(\partial_{j_q} \tilde{A}^i + \partial_i \tilde{A}^{j_q}) \partial^k_{i j_1 \ldots \hat{j}_q \ldots j_k} \tilde{\bfv} \\
= & \ k \sum_{i, j_1, \ldots, j_{k} = 1}^n (\partial^k_{j_1 \ldots j_{k} } \tilde{\bfv}^{\sf T}) (\partial_{j_k} \tilde{A}^i 
  + \partial_i \tilde{A}^{j_k}) \partial^k_{i j_1 \ldots j_{k-1} } \tilde{\bfv} .
\end{align*}
Thus the summed equation can be written as
\begin{multline}
\sum_{j_1, \ldots, j_k=1}^n \tilde{Q}_0(\partial^k_{j_1, \ldots, j_k} \tilde{\bfv}) 
+ k \sum_{j_1, \ldots, j_{k-1}=1}^n \tilde{Q}_1(\partial^k_{j_1, \ldots, j_{k-1},1} \tilde{\bfv}, \ldots, \partial^k_{j_1, \ldots, j_{k-1},n}\tilde{\bfv})  \\
= \sum_{j_1, \ldots, j_k=1}^n \Bigg{(} 2  \partial^k_{j_1\ldots j_k} \tilde{\bfv}^{\sf T} 
  \Big(\partial^k_{j_1\ldots j_k} \tilde{\mathbf{h}} - ( \partial^k_{j_1\ldots j_k} \tilde{B}) \tilde{\bfv} \Big)  \qquad \qquad \qquad \qquad \qquad \\
\qquad \qquad  - \sum_{m=1}^{k-1}  \partial^k_{j_1\ldots j_k} \tilde{\bfv}^{\sf T} \Big(\sum_{i=1}^n D^{k+1-m}  \tilde{A}^i + D^{k-m} \tilde{B}  \Big)   D^{m}  \tilde{\bfv}  \\
 - \sum_{i=1}^n \partial_i\left( \partial^k_{j_1, \ldots, j_{k}} \tilde{\bfv}^{\sf T} \tilde{A}^i \partial^k_{j_1, \ldots, j_{k}} 
    \tilde{\bfv}\right) \Bigg{)}. \qquad \qquad \qquad\qquad\qquad\qquad\,\,\,\quad
 \label{estimate-k-step-4}
\end{multline}
By Proposition \ref{extended-ssp-prop}, the left-hand side of equation \eqref{estimate-k-step-4} is bounded below by
\[ \tfrac{1}{2} (\lambda_0 + k \lambda_1) \sum_{j_1, \ldots, j_k=1}^n | \partial^k_{j_1, \ldots, j_k} \tilde{\bfv}|^2. \]
Thus, after performing operations similar to those above, we obtain
\begin{equation}\label{estimate-k-step-4a}
\begin{aligned}
\|\tilde{\mathbf{v}}\|_k^2 & \le C_k(\lambda_0, \lambda_1)^2 \left(\|\tilde{\mathbf{h}}\|_k^2+
\|\tilde{\bfv}\|_0^2 \|\tilde{B}\|^2_{k,\infty} +
\sum_{m=1}^{k-1} \| \tilde{\bfv} \|_m^2 \Big( \sum_{j=1}^n \|\tilde{A}^j\|^2_{k+1-m,\infty} +\|\tilde{B}\|^2_{k-m,\infty} \Big) \right) \\
& \le C_k(\lambda_0, \lambda_1)^2 \left(\|\tilde{\mathbf{h}}\|^2_k+
\sum_{m=0}^{k-1} \| \tilde{\bfv} \|^2_m \Big( \sum_{j=1}^n \|\tilde{A}^j\|^2_{k+1-m,\infty} +\|\tilde{B}\|^2_{k-m,\infty} \Big) \right).
\end{aligned}
\end{equation}
By the Sobolev embedding estimate \eqref{Sobolev-embedding-estimate}, we have
\[ \|\tilde{A}^j\|_{k+1-m,\infty} \leq K_{k} \|\tilde{A}^j\|_{k+2-m+[\frac{n}{2}]}, \qquad
\|\tilde{B}\|_{k-m,\infty} \leq K_{k} \|\tilde{B}\|_{k+1-m+[\frac{n}{2}]} \]
for some constant $K_{k}$; thus we can write the inequality \eqref{estimate-k-step-4a} as
\begin{equation*}
\|\tilde{\mathbf{v}}\|^2_k
\le C_k^2 \left(\|\tilde{\mathbf{h}}\|^2_k+
\sum_{m=0}^{k-1} \| \tilde{\bfv} \|^2_m \Big( \sum_{j=1}^n \|\tilde{A}^j\|^2_{k+2-m+[\frac{n}{2}]} +\|\tilde{B}\|^2_{k+1-m+[\frac{n}{2}]} \Big) \right),
\end{equation*}
and hence
\begin{equation}
\|\tilde{\mathbf{v}}\|_k
\le C_k \left(\|\tilde{\mathbf{h}}\|_k+
\sum_{m=0}^{k-1} \| \tilde{\bfv} \|_m \Big( \sum_{j=1}^n \|\tilde{A}^j\|_{k+2-m+[\frac{n}{2}]} +\|\tilde{B}\|_{k+1-m+[\frac{n}{2}]} \Big) \right).
\label{estimate-k-step-5}
\end{equation}

By the Gagliardo-Nirenberg interpolation inequality \cite{Evans10} and the Cauchy-Schwarz inequality, for $0 \leq m \leq k-1$, we have
\begin{align*}
\|\tilde{\mathbf{v}}\|_{m} \|\tilde{A}^j\|_{k+2-m+[\frac{n}{2}]}
& \le \tilde{C}_{m} (\|\tilde{\mathbf{v}}\|_{0}\|\tilde{A}^j\|_{k+2+[\frac{n}{2}]} +  \|\tilde{\mathbf{v}}\|_{k-1}\|\tilde{A}^j\|_{3+[\frac{n}{2}]})  , \\[0.1in]
\|\tilde{\mathbf{v}}\|_{m} \|\tilde{B}\|_{k+1-m+[\frac{n}{2}]}
& \le \tilde{C}_{m} (\|\tilde{\mathbf{v}}\|_{0}\|\tilde{B}\|_{k+1+[\frac{n}{2}]} +  \|\tilde{\mathbf{v}}\|_{k-1}\|\tilde{B}\|_{2+[\frac{n}{2}]})
\end{align*}
for some constant $\tilde{C}_{m}$.
Substituting into equation \eqref{estimate-k-step-5}, we obtain
\begin{multline}\label{estimate-k-step-7}
\qquad \|\tilde{\mathbf{v}}\|_k
\le
C'_k\left(\|\tilde{\mathbf{h}}\|_{k}  + \|\tilde{\bfv}\|_0\Big( \sum_{j=1}^n \|\tilde{A}^j\|_{k+2+[\frac{n}{2}]} + \|\tilde{B}\|_{k+1+[\frac{n}{2}]} \Big) \right. \\
\left. + \|\tilde{\bfv}\|_{k-1} \Big( \sum_{j=1}^n \|\tilde{A}^j\|_{3+[\frac{n}{2}]} + \|\tilde{B}\|_{2+[\frac{n}{2}]} \Big) \right)
. \qquad \qquad \qquad\qquad\qquad\,\,\quad
\end{multline}

Now let $\alpha \geq 4 + [\frac{n}{2}]$.  It follows from the fact that $A$ and $B$ are $C^\infty$ functions of $\bfu$ 
and its first derivatives that there exist constants $\tilde{K}_{\rho}$ and $\tilde{K}_{k,\rho}$ such that, 
for any $\bfu \in D_\alpha$, the extended linear system \eqref{extended-linear-sys} 
corresponding to the linearization of \eqref{nonlinear-sys} at $\bfu$ satisfies
\[ 
\|\tilde{A}^j\|_{3+[\frac{n}{2}]}, \|\tilde{B}\|_{2+[\frac{n}{2}]} \leq \tilde{K}_{\rho}, \qquad
 \|\tilde{A}^j\|_{k+2+[\frac{n}{2}]}, \|\tilde{B}\|_{k+1+[\frac{n}{2}]} \leq \tilde{K}_{k,\rho}(1 +  \|\bfu\|_{k+3+[\frac{n}{2}]}).
  \]
Thus \eqref{estimate-k-step-7} becomes
\begin{equation}\label{estimate-k-step-8}
\|\tilde{\mathbf{v}}\|_k \leq C''_k \left( \|\tilde{\mathbf{h}}\|_{k}  + \|\tilde{\bfv}\|_0\|\bfu\|_{k+3+[\frac{n}{2}]} + \|\tilde{\bfv}\|_{k-1} \right).
\end{equation}
It then follows by induction (with the inequality \eqref{v-order-1} as the base case) that
\begin{equation}\label{estimate-k-step-9}
\|\tilde{\mathbf{v}}\|_k \leq C'''_k \left( \|\tilde{\mathbf{h}}\|_{k}  + \|\tilde{\bfv}\|_0\|\bfu\|_{k+3+[\frac{n}{2}]}  \right).
\end{equation}
Therefore, the restriction $\bfv$ of $\tilde{\bfv}$ to $B_r$ satisfies
\begin{equation}\label{v-order-k}
\begin{aligned}
\|\bfv\|_k \leq \|\tilde{\bfv}\|_k & \leq C'''_k \left( \|\tilde{\mathbf{h}}\|_{k}  + \|\tilde{\bfv}\|_0\|\bfu\|_{k+3+[\frac{n}{2}]}  \right) \\
& \leq C'''_k M_{k,2} \left( \|\mathbf{h}\|_{k}  + \|\bfv\|_0\|\bfu\|_{k+3+[\frac{n}{2}]}  \right) \\
& \leq \tilde{C}_k \left( \|\mathbf{h}\|_{k}  + \|\bfh\|_0\|\bfu\|_{k+3+[\frac{n}{2}]}  \right) .
\end{aligned}
\end{equation}

All the hypotheses of Theorem \ref{Nash-Moser-theorem} have now been verified for any $\alpha \geq 3 + [\frac{n}{2}]$; 
thus the conclusion of Theorem \ref{Nash-Moser-theorem} gives the desired solution $\bfu \in C^\infty(B_r, \R^s)$ 
to the nonlinear system \eqref{nonlinear-sys} on $B_r$.  This completes the proof of Theorem \ref{nonlinear-theorem}.

\part{Application to Isometric Embedding}

\section{Local existence theorems for isometric embedding}\label{state-main-theorems-sec}

The remainder of this paper will be devoted to giving a new proof, based on Theorem \ref{nonlinear-theorem}, for the following local existence theorem:

\begin{ourtheorem}\label{main-theorem-3D}
Let $(M,g)$ be a $C^\infty$ Riemannian manifold of dimension $n=2$ or $n=3$, let $N = \tfrac{1}{2}n(n+1)$, 
let $\bfx_0 \in M$ so that the Riemann curvature tensor $R(\bfx_0)$ is nonzero.
Then there exists a neighborhood $\Omega \subset M$ of $\bfx_0$ 
for which there is a $C^\infty$ isometric embedding $\bfy:\Omega \to \R^N$.
\end{ourtheorem}

Here we briefly describe our strategy for proving Theorem \ref{main-theorem-3D}.  
Let $n=2$ or $n=3$, and let $N = \tfrac{1}{2}n(n+1)$.  
For convenience, choose local coordinates $\bfx = (x^1, \ldots, x^n)$ based at $\bfx_0$, 
so that without loss of generality we may assume that $\bfx_0 = \mathbf{0}$.  
Given a $C^\infty$ metric $g$ on a neighborhood $\Omega$ of $\bfx = \mathbf{0}$, 
choose a real analytic metric $\bar{g}$ on $\Omega$ that agrees with $g$ to sufficiently high order at $\bfx = \mathbf{0}$.  
By the Cartan-Janet theorem, there exists a real analytic isometric embedding (possibly on a smaller neighborhood) 
$\bfy_0: \Omega \to \R^N$ of $(\Omega, \bar{g})$ into $\R^N$.

The linearization of the isometric embedding system \eqref{equation:isometric} at $\bfy_0$ is a first-order PDE system 
of $N$ equations for the unknown function $\bfv:\Omega \to \R^N$.  
This system decomposes into a system of $n$ first-order PDEs for the tangential components of $\bfv$, 
together with $(N-n)$ equations that determine the normal components of $\bfv$ algebraically in terms of the tangential components.

We will show that, under the hypotheses of Theorem \ref{main-theorem-3D}, 
the embedding $\bfy_0$ can be chosen so that the tangential subsystem becomes strongly symmetric positive after a fairly simple, 
but carefully chosen, change of variables.  Consequently, it follows from the argument given in the proof 
of Theorem \ref{nonlinear-theorem} that the tangential components of $\bfv$ satisfy the smooth tame estimates required to implement 
a Nash-Moser iteration scheme for the isometric embedding system \eqref{equation:isometric}, 
and then the remaining algebraic equations will imply the necessary estimates for the normal components of $\bfv$.  
Theorem \ref{main-theorem-3D} then follows directly from the Nash-Moser implicit function theorem (cf. Theorem \ref{Nash-Moser-theorem}).

\begin{notation}
We will use the Einstein summation convention for the remainder of this paper.
\end{notation}

\section{The linearized isometric embedding system and Nash-Moser iteration}\label{linearization-sec}

Let $\Omega \subset \R^n$ be a neighborhood of $\bfx = \mathbf{0}$.  Let $\bfy_0:\Omega \to \R^N$ be a smooth embedding, 
and let $\bar{g} = \bar{g}_{ij} dx^i dx^j$ be the metric on $\Omega$ induced by the restriction of the Euclidean metric on $\R^N$ to $\bfy_0(\Omega)$.
Linearization of the isometric embedding system \eqref{equation:isometric} at the function $\bfy_0$ yields the linear PDE system
\begin{equation}\label{equation:linearized}
\partial_i\bfy_0\cdot\partial_j\bfv + \partial_j\bfy_0\cdot\partial_i\bfv = h_{ij}, \qquad 1 \leq i,j \leq n,
\end{equation}
for the function $\bfv:\Omega \to \R^N$, where $h_{ij} = g_{ij} - \bar{g}_{ij}$.

As described in \cite{BGY83},
the linearized system \eqref{equation:linearized} can be reformulated as a system of $n$ linear PDEs for the $n$ tangential components of $\bfv$, 
together with a system of $(N-n)$ algebraic equations for the normal components.  To this end, note that,
since $\bfy_0$ is an embedding, for each $\bfx \in \Omega$ the tangent vectors 
$\{\partial_1\bfy_0(\bfx), \ldots, \partial_n\bfy_0(\bfx)\}$ are linearly independent 
and span an $n$-dimensional subspace $T_{\bfx} \subset \R^N$. We can therefore decompose the second derivatives of $\bfy_0$ as follows:
\begin{equation}
\partial^2_{ij}\bfy_0 = \Gamma_{ij}^k\partial_k\bfy_0 + H_{ij}, \label{second-derivatives-decomp}
\end{equation}
where, for each $1 \le i, j \le n$, the vector-valued function $H_{ij} = H_{ji}: \Omega \to \R^N$ satisfies 
$H_{ij}\cdot\partial_k\bfy_0 = 0$ for $1 \le k \le n$.  
The functions $\Gamma^k_{ij}:\Omega \to \R$ are the Christoffel symbols of the metric $\bar{g}_{ij}$, 
and the quadratic form $H_{ij} dx^i dx^j$ is the second fundamental form of the embedding $\bfy_0$.

Let $\mathcal{S}_n$ denote the $\frac{1}{2}n(n+1)$-dimensional space of quadratic forms on $\R^n$, represented by symmetric $n\times n$ matrices $[s_{ij}]$.
For each $\bfx \in \Omega$, the vectors $H_{ij}(\bfx)$ determine a linear map $H_{\bfx}: \R^N \rightarrow \mathcal{S}_n$, given by
\[ H_{\bfx}(\mathbf{v}) = [\langle H_{ij}(\bfx), \mathbf{v} \rangle ]. \]
We denote the image by $\II_{\bfx} = H_{\bfx}(\R^N)$. Since the kernel of the map $H_{\bfx}$ contains $T_{\bfx}$, we have
\begin{equation}\label{equation:dim-II}
\dim \II_{\bfx} \le \frac{1}{2}n(n+1) - n = \frac{1}{2}n(n-1).
\end{equation}

\begin{definition}
The embedding $\bfy_0:\Omega \to \R^N$ is called {\em nondegenerate} if $\dim \II_{\bfx} = \frac{1}{2}n(n-1)$ for all $\bfx \in \Omega$.
\end{definition}

Now, let $\mathcal{S}^*_n$ denote the dual space to $\mathcal{S}_n$, represented by symmetric matrices $[s^{ij}]$, 
with the pairing $ \mathcal{S}^*_n \times  \mathcal{S}_n \to \R$ defined for
$A \in \mathcal{S}^*_n, H \in \mathcal{S}_n$ by
\begin{equation}
\langle A, H\rangle = \sum_{i,j=1}^n A^{ij}H_{ij}. \label{define-pairing}
\end{equation}

\begin{definition}
The {\em annihilator} $\II_{\bfx}^\perp$ of the subspace $\II_{\bfx} \subset \mathcal{S}_n$ is the subspace of $\mathcal{S}^*_n$ defined by
\[
\II_{\bfx}^\perp = \{ A \in \mathcal{S}^*_n\ :\ \langle A, H\rangle = 0 \text{ for all }H \in \II_{\bfx}\}.
\]
\end{definition}
It follows from equation \eqref{equation:dim-II} that
$\dim \II_{\bfx}^\perp \ge n$, with equality for all $\bfx \in \Omega$ if and only if $\bfy_0$ is nondegenerate.
\begin{assumption}\label{nondegeneracy-assumption}
Henceforth, we will assume that $\bfy_0$ is nondegenerate, 
and that consequently  $\dim \II_{\bfx} = \frac{1}{2}n(n-1)$ and  $\dim \II_{\bfx}^\perp = n$ for all $\bfx \in \Omega$.
\end{assumption}

Now, the system
\eqref{equation:linearized} can be rewritten as follows:
\begin{equation}\label{equation:parts}
\partial_i(\partial_j\bfy_0\cdot \bfv) + \partial_j(\partial_i\bfy_0\cdot \bfv)
- 2\bfv\cdot\partial^2_{ij}\bfy_0 = h_{ij}, \qquad 1 \leq i,j \leq n.
\end{equation}
Define functions $\bar{v}_i$ and $\nabla_j\bar{v}_i$ by
\begin{alignat*}{2}
\bar{v}_i &= \bfv\cdot\partial_i \bfy_0, &\qquad 1 &\leq i \leq n,\\
\nabla_j \bar{v}_i &= \partial_j \bar{v}_i - \Gamma_{ij}^k \bar{v}_k, &\qquad 1 &\leq i,j \leq n;
\end{alignat*}
then the system \eqref{equation:parts} can be written as
\begin{equation}\label{equation:rewritten}
\nabla_i \bar{v}_j + \nabla_j \bar{v}_i - 2\bfv\cdot H_{ij} = h_{ij}, \qquad 1 \leq i,j \leq n.
\end{equation}

Since $\dim \II_{\bfx}^\perp = n$,
there must exist smooth maps $A^1, \ldots, A^n: \Omega \rightarrow \mathcal{S}^*_n$
such that, for each $\bfx \in \Omega$, the matrices
$A^1(\bfx), \dots, A^n(\bfx)$ comprise a basis of $\II_{\bfx}^\perp$.
By writing $A^k = [A^{kij}]$ and pairing each of these matrices with equations \eqref{equation:rewritten} as in \eqref{define-pairing}, we
obtain the following system of $n$ first order PDEs for the functions $\bar{v}_1, \ldots, \bar{v}_n$:
\begin{equation}\label{equation:reduced}
A^{kij}(\nabla_i \bar{v}_j + \nabla_j \bar{v}_i) = A^{kij}h_{ij}, \qquad k = 1, \dots, n.
\end{equation}
Because $A^k \in \mathcal{S}^*_n$, we have $A^{kij} = A^{kji}$, but the component functions $A^{kij}$
do not necessarily possess any other symmetries.

\begin{proposition}\label{reduced-system-prop}
Any solution $(\bar{v}_1, \ldots, \bar{v}_n):\Omega \to \R^n$ to \eqref{equation:reduced} uniquely determines a solution $\bfv:\Omega \to \R^N$
to equation \eqref{equation:linearized}{\rm ;} moreover, $\bfv$ can be determined algebraically from $(\bar{v}_1, \ldots, \bar{v}_n)$.
\end{proposition}

\begin{proof}
Suppose that $\bar{v}_1, \dots, \bar{v}_n$ satisfy
\eqref{equation:reduced}, and define
\begin{equation}\label{define-etaij}
\eta_{ij} = h_{ij} - \nabla_i \bar{v}_j - \nabla_j \bar{v}_i, \qquad 1 \leq i,j \leq n.
\end{equation}
Equation \eqref{equation:reduced} implies that $[\eta_{ij}(\bfx)] \in \II_{\bfx}$ for
each $\bfx \in \Omega$.  Assumption \ref{nondegeneracy-assumption} implies that $H_{\bfx}$ has maximal rank so that,
for each $\bfx \in \Omega$, there exists a unique $\bfv(\bfx) \in \R^N$ such that
\begin{equation}
\begin{aligned}
\bfv(\bfx) \cdot\partial_i\bfy_0(\bfx) &= \bar{v}_i(\bfx) , &\qquad 1 &\leq i \leq n,
\\
-2 \bfv(\bfx) \cdot H_{ij}(\bfx) &= \eta_{ij}(\bfx), &\qquad 1 &\leq i,j \leq n.
\end{aligned} \label{alg-sys-for-u}
\end{equation}
Therefore, the map $\bfv: \Omega \rightarrow \R^N$ satisfies
\eqref{equation:rewritten}, which in turn is equivalent to
\eqref{equation:linearized}.
\end{proof}

It follows from Proposition \ref{reduced-system-prop} that, in order to solve the linearized equations \eqref{equation:linearized}, 
it suffices to solve equations \eqref{equation:reduced}.
This system can be written as
\begin{equation}
A^{kij} (\partial_i \bar{v}_j + \partial_j \bar{v}_i - 2\Gamma^{\ell}_{ij} \bar{v}_{\ell}) = A^{kij} h_{ij}, \qquad k = 1, \dots, n. 
\label{equation:reduced-2a}
\end{equation}
Since $A^{kij} = A^{kji}$ and $\Gamma^k_{ij} = \Gamma^k_{ji}$, this is equivalent to the system
\begin{equation}
A^{kij} (\partial_i \bar{v}_j  - \Gamma^{\ell}_{ij} \bar{v}_{\ell}) = \tfrac{1}{2}A^{kij} h_{ij}, \qquad k = 1, \dots, n. \label{equation:reduced-2}
\end{equation}
We can write this system in matrix form as follows: For $i = 1, \dots, n$, let $\bar{A}^i$ denote the matrix
\[ \bar{A}^i = [A^{kij}] =  \begin{bmatrix} A^{1i1} & \cdots & A^{1in} \\[0.1in] \vdots & & \vdots \\[0.1in] A^{ni1} & \cdots & A^{nin} \end{bmatrix}. \]
Then the system \eqref{equation:reduced-2} can be written as
\begin{equation}
\bar{A}^i \partial_i \bar{\bfv} + B \bar{\bfv} = \bfh, \label{equation:unsymmetrized-system}
\end{equation}
where
\begin{equation}
 \bar{\bfv} = [ \bar{v}_j ], \qquad
 B = [B^{kj}] = [ -A^{k\ell m} \Gamma^{j}_{\ell m} ], \qquad \bfh = [\tfrac{1}{2}A^{k\ell m} h_{\ell m} ], \qquad 1 \leq j,k,\ell, m \leq n. \label{define-matrices}
\end{equation}

Our proof of Theorem \ref{main-theorem-3D} is based on the following key result.

\begin{proposition}\label{key-ssp-prop}
Suppose that the system \eqref{equation:unsymmetrized-system} is strongly symmetric positive at $\bfx = \mathbf{0}$.  
Then there exist a neighborhood $\Omega_0 \subset \Omega$ of $\bfx = \mathbf{0}$, an integer $\beta$, and $\epsilon > 0$ such that, 
for any $C^\infty$ metric $g$ on $\Omega_0$ with $\|g - \bar{g}\|_\beta < \epsilon$, 
there exists a $C^\infty$ solution $\bfy:\Omega_0 \to \R^N$ to the isometric embedding system \eqref{equation:isometric}.

Moreover, the conclusion holds if the system \eqref{equation:unsymmetrized-system} becomes strongly symmetric positive after performing a change of variables of the form
\begin{equation}
 \bar{\bfx} = \phi(\bfx), \qquad \bar{\bfw} = S(\bfx)\bar{\bfv}, \label{general-change-of-vars}
\end{equation}
where $\phi:\Omega \to \R^n$ is a local diffeomorphism of $\Omega$ with $\phi(\mathbf{0}) = \mathbf{0}$, 
and $S:\Omega \to \R^{n \times n}$ is a $C^\infty$, $n\times n$ matrix-valued function on $\Omega$ with $S(\mathbf{0})$ invertible.
\end{proposition}

\begin{proof}
First, suppose that the system \eqref{equation:unsymmetrized-system} is strongly symmetric positive at $\bfx = \mathbf{0}$.
The argument from the proof of Theorem \ref{nonlinear-theorem} shows that, under the hypotheses of the proposition, 
there exists a neighborhood $\Omega_0 \subset \Omega$ of $\bfx = \mathbf{0}$ on which the system \eqref{equation:unsymmetrized-system}
corresponding to the linearization of \eqref{equation:isometric} at any function $\bfy:\Omega_0 \to \R^N$ sufficiently close to $\bfy_0$ has a solution $\bar{\bfv}$
that satisfies the estimates of the form:
\begin{equation}
 \|\bar{\bfv} \|_k \leq C'_k \left( \|\mathbf{h}\|_{k}  + \|\bfh\|_0\|\bfy - \bfy_0\|_{k+3+[\frac{n}{2}]}  \right) , \qquad k \geq 0, \label{subsystem-estimates}
\end{equation}
for some constants $C'_k$.
Then it follows from equation \eqref{define-etaij} that
\[
\|\eta\|_k \leq C''_k \left( \|\mathbf{h}\|_{k+1}  + \|\bfh\|_0\|\bfy - \bfy_0\|_{k+4+[\frac{n}{2}]}  \right) , \qquad k \geq 0,
\]
for some constants $C''_k$.
These estimates, together with equations \eqref{alg-sys-for-u} and Assumption \ref{nondegeneracy-assumption}, imply (possibly after shrinking $\Omega_0$) that
\begin{equation}
 \| \bfv \|_k \leq C_k \left( \|\mathbf{h}\|_{k+1}  + \|\bfh\|_0\|\bfy - \bfy_0\|_{k+4+[\frac{n}{2}]}  \right) , \qquad k \geq 0 \label{full-system-estimates}
\end{equation}
for some constants $C_k$.  The existence of a solution $\bfy:\Omega_0 \to \R^N$ to the system \eqref{equation:isometric} then 
follows from Theorem \ref{Nash-Moser-theorem}, just as in the proof of Theorem \ref{nonlinear-theorem}.

For the second statement, assume that $\Omega_0$ has been chosen so that the restriction of $\phi$ to $\Omega_0$
is smoothly invertible and the matrix $S(\bfx)$ is invertible for all $\bfx \in \Omega_0$, with the determinant 
of $S(\bfx)$ bounded away from $0$.  Then it suffices to observe that a change of coordinates of the form \eqref{general-change-of-vars} 
induces linear maps $\psi_k: H^k(\Omega_0) \to H^k(\phi(\Omega_0))$ defined by
\[
\psi_k(\bar{\bfv}) = \bar{\bfw},
\]
and that these maps are continuous with continuous inverse.
Thus the estimates of the form \eqref{subsystem-estimates}
for the function $\bar{\bfw}$ imply similar estimates for $\bar{\bfv}$, which in turn imply the estimates \eqref{full-system-estimates} for $\bfv$.
\end{proof}

Thus it remains to show that, under the hypotheses of Theorem \ref{main-theorem-3D}, the approximate embedding $\bfy_0:\Omega \to \R^N$ 
can be chosen so that the linearized system \eqref{equation:unsymmetrized-system} becomes strongly symmetric positive at $\bfx = \mathbf{0}$ 
after a change of variables of the form \eqref{general-change-of-vars}.

\section{Symmetrization}

The matrices $\bar{A}^i$ in the system \eqref{equation:unsymmetrized-system} are not necessarily symmetric, 
because the functions $A^{kij}$ and $A^{jik}$ are not necessarily equal.  
The system \eqref{equation:unsymmetrized-system} can be re-expressed as a symmetric system 
if and only if there exists an invertible $n \times n$ matrix $C$ such that the matrices
\[ C\bar{A}^1, \dots, C\bar{A}^n \]
are all symmetric, in which case multiplying the system \eqref{equation:unsymmetrized-system} by $C$ results in a symmetric system.

Observe that multiplying \eqref{equation:unsymmetrized-system} by an invertible matrix $C$ is equivalent to replacing 
the basis $A^1, \ldots, A^n$ for the annihilator $\II_{\bfx}^{\perp}$ at each point with the alternate basis
\[ A'^{\ell} = C^{\ell}_k A^k . \]
Moreover, a given basis $A^1, \dots, A^n$ for $\II_{\bfx}^{\perp}$ will lead to symmetric matrices $\bar{A}^1, \dots, \bar{A}^n$ if and only if
\begin{equation}
 A^{kij} = A^{jik}; \label{equation:symmetry-condition}
\end{equation}
i.e., if and only if the coefficients $A^{ijk}$ are symmetric in all their indices.  
Therefore, in order to determine whether the system \eqref{equation:unsymmetrized-system} is symmetrizable, 
it suffices to determine whether there exists a basis
$A^k = [A^{kij}]$
for $\II_{\bfx}^{\perp}$ for which the coefficients $A^{kij}$ are symmetric in all their indices.  
If we choose such a basis for $\II_{\bfx}^{\perp}$, then we will have
\[ \bar{A}^k = A^k, \]
and there will be no need to distinguish between the two.

\begin{proposition}\label{symmetrization-prop}
When $n=2$ or $n=3$, the linearized system \eqref{equation:unsymmetrized-system} is symmetrizable.
\end{proposition}

\begin{proof}
When $n=2$, we have $N = 3$.
Choose any smoothly varying basis element $H^3(\bfx) \in \II_{\bfx}$.  
Consider the 4-dimensional space of all symmetric cubic forms
\[ 
A = A^{ijk} \left(\frac{\partial}{\partial x^i} \circ \frac{\partial}{\partial x^j}\circ \frac{\partial}{\partial x^k} \right) \in S^3(T\R^2), 
\]
and for $k=1,2$, let $A^k$ denote the matrix $A^k = [A^{kij}]$.  The annihilator equations
\[
\langle A^k, H^3 \rangle = 0, \qquad k=1,2,
\]
form a system of $2$ homogeneous linear equations for the 4 functions $A^{kij}$.  Thus there must be at least a $2$-dimensional solution 
space at each point $\bfx \in U$, and choosing $A(\bfx)$ to be any smoothly varying, nonvanishing element of this space 
produces a symmetric linearized system \eqref{equation:unsymmetrized-system}.

When $n=3$, we have $N = 6$.
Choose any smoothly varying basis $(H^4(\bfx), H^5(\bfx), H^6(\bfx))$ for the space $\II_{\bfx}$.  Consider the 10-dimensional space of all symmetric cubic forms
\[ 
A = A^{ijk} \left(\frac{\partial}{\partial x^i} \circ \frac{\partial}{\partial x^j}\circ \frac{\partial}{\partial x^k} \right) \in S^3(T\R^3), 
\]
and for $k=1,2,3$, let $A^k$ denote the matrix $A^k = [A^{kij}]$.  The annihilator equations
\[
\langle A^k, H^\alpha \rangle = 0, \qquad k=1,2,3, \ \ \alpha=4,5,6,
\]
form a system of $9$ homogeneous linear equations for the 10 functions $A^{kij}$.  Thus there must be at least a $1$-dimensional 
solution space at each point $\bfx \in U$, and choosing $A(\bfx)$ to be any smoothly varying, nonvanishing element 
of this space produces a symmetric linearized system \eqref{equation:unsymmetrized-system}.
\end{proof}

\begin{remark}
The result of Proposition \ref{symmetrization-prop} does not hold for a generic choice of $\II_\bfx$ when $n \geq 4$; 
this is the primary obstruction to applying our methods to the isometric embedding problem in higher dimensions.
\end{remark}

For the remainder of this paper, we will restrict to the cases $n=2$ and $n=3$. We will assume that the functions $A^{kij}$ are symmetric 
in all their indices, so that the matrices $\bar{A}^i$ in the linear system \eqref{equation:unsymmetrized-system} are symmetric and 
may be identified with the matrices $A^i$.
We will use the convention that Roman indices ($i,j,k$, etc.) range from $1$ to $n$, 
while Greek indices ($\alpha, \beta, \gamma$, etc.) range from $(n+1)$ to $N = \tfrac{1}{2}n(n+1)$.

\section{Compatibility equations and normal forms}

In this section, we will show how the Gauss and Codazzi equations (also called the ``compatibility equations") 
for the embedding $\bfy_0:\Omega \to \R^N$ introduce constraints on the values of the matrices $A^i$ (now assumed to be symmetric) 
and their first derivatives at $\bfx = \mathbf{0}$, and
we will show how the matrices $A^i$ can be put into a simple normal form at the point $\bfx = \mathbf{0}$.

Let $\bfx = (x^1, \ldots, x^n)$ be local coordinates on $\Omega$ centered at $\bfx = \mathbf{0}$.  
We will assume that $\bfx$ is a {\em normal} coordinate system at $\mathbf{0}$ with respect to the metric $g$ on $\Omega$,
i.e., that $\Gamma^k_{ij}(\mathbf{0}) = 0$ for $1 \leq i,j,k \leq n$.  
We will {\em not}, however, assume that $g_{ij}(\mathbf{0}) = \delta_{ij}$, because our argument will involve a nontrivial $GL(n,\R)$ action 
on the tangent space $T_{\mathbf{0}}M$. The specific values of $g_{ij}(\mathbf{0})$ will not affect our argument, in any case.

Let $\bar{g}$ be a real analytic metric on $\Omega$ that agrees with $g$ up to order 
at least $\beta$ (where $\beta$ is as in Proposition \ref{key-ssp-prop}) at $\bfx = \mathbf{0}$, 
and note that this implies that the Riemann curvature tensors of $g$ and $\bar{g}$ agree up to order at least $(\beta-2)$ at $\bfx = \mathbf{0}$.  
By the Cartan-Janet isometric embedding theorem \cite{Jacobowitz82}, 
there exists a real analytic isometric embedding (possibly on a smaller neighborhood) $\bfy_0: \Omega \to \R^N$ of $(\Omega, \bar{g})$ into $\R^N$.

Let $(\e_{n+1}, \ldots, \e_N)$ be a smoothly varying orthonormal basis for the normal bundle 
of the embedded submanifold $\bfy_0(\Omega) \subset \R^N$, chosen so that
\begin{equation}
   \nabla^\perp_{\bfw} \e_{\alpha} (\mathbf{0}) = \mathbf{0} \label{equation:normal-normal-bundle}
\end{equation}
for $n+1 \leq \alpha \leq N$ and all $\bfw \in T_{\mathbf{0}} M$, 
where $\nabla^\perp$ denotes the connection on the normal bundle induced by the Euclidean connection on $\R^N$.  
This condition is the analog for the normal bundle of the normal coordinates condition $\Gamma^k_{ij}(\mathbf{0}) = 0$.  
Then we can write the second fundamental form of $\bfy_0$ as
\begin{equation}
 H_{ij} dx^i \circ dx^j = \e_{\alpha} \otimes H^\alpha_{ij} dx^i \circ dx^j 
 \label{equation:second-fundamental-form-3-d}
\end{equation}
for scalar-valued functions $H^\alpha_{ij}: \Omega \to \R$.

The embedding $\bfy_0:\Omega \to \R^N$ must satisfy the following conditions at $\bfx=\mathbf{0}$:
\begin{itemize}
 \item Metric conditions:
 \begin{gather*}
  (\partial_i \bfy_0 \cdot \partial_j \bfy_0)\vert_{\bfx = \mathbf{0}} = g_{ij}(\mathbf{0}), \ \ \ 1 \leq i,j \leq n, \\
 \nabla_i (\partial_j \bfy_0)\vert_{\bfx = \mathbf{0}}  = (\Gamma^k_{ij} \partial_k \bfy_0)\vert_{\bfx = \mathbf{0}} = \mathbf{0}, \ \ \ 1 \leq i,j,k \leq n;
\end{gather*}
\item Gauss equations and their first derivatives:
\begin{align}
 \left(\sum_{\alpha = n+1}^N (H^{\alpha}_{ik} H^{\alpha}_{j\ell} - H^{\alpha}_{i\ell} H^{\alpha}_{jk}) \right)
 \bigg{\vert}_{\bfx = \mathbf{0}} & = R_{ijk\ell}(\mathbf{0}), \ \ \  1 \leq i,j,k, \ell \leq n, \label{Gauss-eq-at-0} \\
 \partial_m \left( \sum_{\alpha = n+1}^N (H^{\alpha}_{ik} H^{\alpha}_{j\ell} - H^{\alpha}_{i\ell} H^{\alpha}_{jk})  \right) 
 \bigg{\vert}_{\bfx = \mathbf{0}} & = (\partial_m R_{ijk\ell}) (\mathbf{0}),  \ \ \  1 \leq i,j,k, \ell, m \leq n,  \label{d-Gauss-eq-at-0}
\end{align}
where $R_{ijk\ell}$ denotes the components of the Riemann curvature tensor of $(M,g)$;
\item Codazzi equations:
\begin{equation}
 (\partial_i H^\alpha_{jk})\big{\vert}_{\bfx = \mathbf{0}} =
(\partial_j H^\alpha_{ki})\big{\vert}_{\bfx = \mathbf{0}} =
(\partial_k H^\alpha_{ij})\big{\vert}_{\bfx = \mathbf{0}}, \ \ \ 1 \leq i,j,k \leq n, \ \ \ n+1 \leq \alpha \leq N. \label{Codazzi-eq-at-0}
\end{equation}
This form of the Codazzi equations at $\bfx = \mathbf{0}$ relies on the normal coordinates 
condition $\Gamma^i_{jk}(\mathbf{0}) = 0$ and the condition \eqref{equation:normal-normal-bundle} on the covariant derivatives of $\e_{\alpha}$.
\end{itemize}
Conversely, the Cartan-Janet theorem guarantees that, for any choice of real numbers $H^\alpha_{ij}(\mathbf{0})$ 
and $\partial_k H^\alpha_{ij}(\mathbf{0})$ satisfying equations \eqref{Gauss-eq-at-0}--\eqref{Codazzi-eq-at-0}, 
there exists a real analytic isometric embedding $\bfy_0:\Omega \to \R^N$ of $(\Omega, \bar{g})$ (possibly after shrinking $\Omega$) 
whose second fundamental form agrees with the given values up to first order at $\bfx = \mathbf{0}$.

\begin{notation}
Henceforth, we will only be concerned with the values of $H^\alpha_{ij}$, $A^{kij}$, and their first derivatives at $\bfx = \mathbf{0}$.  Thus we will use the following notations:
\begin{itemize}
\item $H^\alpha_{ij}$ will denote the real number $H^\alpha_{ij}(\mathbf{0})$, and $H^\alpha$ will denote the matrix $[H^\alpha_{ij}]$.

\item $h^\alpha_{kij}$ will denote the real number $\partial_k H^\alpha_{ij}(\mathbf{0})$, and $h^\alpha_k$ will denote the matrix $[h^\alpha_{kij}]$.  
Note that the Codazzi equations \eqref{Codazzi-eq-at-0} are equivalent to the condition that the $h^\alpha_{kij}$ are fully symmetric in their lower indices.

\item $A^{kij}$ will denote the real number $A^{kij}(\mathbf{0})$, and $A^k$ will denote the matrix $[A^{kij}]$.

\item $a^{kij}_\ell$ will denote the real number $\partial_\ell A^{kij}(\mathbf{0})$, and $a^k_{\ell}$ will denote the matrix $[a^{kij}_\ell]$.  
Note that the $a^{kij}_\ell$ are fully symmetric in their upper indices, but there are no symmetries involving the lower index.

\item $R_{ijk\ell}$ will denote the real number $R_{ijk\ell}(\mathbf{0})$.  Note that the $R_{ijk\ell}$ must satisfy the symmetries of the Riemann curvature tensor:
\[ R_{ijk\ell} = -R_{jik\ell} = -R_{ij\ell k} = R_{k\ell ij}. \]
When $n=2$, the only nonzero component of $R$ is the Gauss curvature $K = R_{1212}$; when $n=3$, $R$ has 6 nonzero components, 
represented by $R_{1212}$, $R_{2323}$, $R_{3131}$, $R_{1223}$, $R_{2331}$, $R_{3112}$.

\item $r_{ijk\ell,m}$ will denote the real number $\partial_m R_{ijk\ell}(\mathbf{0})$; when $n=2$, we will denote $r_{1212,1}$ 
and $r_{1212,2}$ by $k_1$ and $k_2$, respectively.   Note that the $r_{ijk\ell,m}$ must satisfy the same symmetries 
as the $R_{ijk\ell}$ in their first four indices, together with the second Bianchi identities.
    When $n=2$, the second Bianchi identities are trivial; when $n=3$, they are represented by the three equations:
\begin{equation}
 r_{2323,1} + r_{2331,2} + r_{1223,3} =
r_{2331,1} + r_{3131,2} + r_{3112,3} =
r_{1223,1} + r_{3112,2} + r_{1212,3} = 0 . \label{second-Bianchi-eqns}
\end{equation}
\end{itemize}
\end{notation}

The values of $H^\alpha_{ij}, h^\alpha_{ijk}, A^{kij}, a^{kij}_\ell$ are constrained by the following relations 
and are otherwise arbitrary (apart from the nondegeneracy condition on the $H^\alpha_{ij}$):
\begin{itemize}
\item {\bf Gauss equations:}
\begin{equation}
\sum_{\alpha = n+1}^N (H^{\alpha}_{ik} H^{\alpha}_{j\ell} - H^{\alpha}_{i\ell} H^{\alpha}_{jk}) = R_{ijk\ell}, \qquad 1 \leq i,j,k, \ell \leq n; \label{Gauss-eqns}
\end{equation}
\item {\bf Codazzi equations:}
\begin{equation}
h^{\alpha}_{ijk} = h^{\alpha}_{ikj} = h^{\alpha}_{jik}, \qquad 1 \leq i,j,k  \leq n, \ \ n+1 \leq \alpha \leq N; \label{Codazzi-eqns}
\end{equation}
\item {\bf Annihilator equations:}
\begin{equation}
A^{kij} H^{\alpha}_{ij} = 0, \qquad 1 \leq k \leq n, \ \ n+1 \leq \alpha \leq N; \label{annihilator-eqns}
\end{equation}
\item {\bf Derivatives of the Gauss equations:}
\begin{equation}
\sum_{\alpha = n+1}^N
( H^{\alpha}_{ik} h^{\alpha}_{j\ell m} +  H^{\alpha}_{j\ell} h^{\alpha}_{ik m} - H^{\alpha}_{i\ell} h^{\alpha}_{jkm} - H^{\alpha}_{jk} h^{\alpha}_{i\ell m}) 
= r_{ijk\ell,m}, \qquad 1 \leq i,j,k, \ell, m \leq n;
  \label{d-Gauss-eqns}
\end{equation}
\item {\bf Derivatives of the annihilator equations:}
 \begin{equation}
 A^{kij} h^{\alpha}_{ij\ell} + H^{\alpha}_{ij} a^{kij}_{\ell} = 0, \qquad 1 \leq k, \ell \leq n, \ \ n+1 \leq \alpha  \leq N. \label{d-annihilator-eqns}
\end{equation}
\end{itemize}

It will be helpful to reduce to the case where the values $H^\alpha_{ij}$ and $A^{kij}$ take on relatively simple normal forms.
To this end, consider a linear transformation of the independent variables of the form:
\begin{equation}
 \bfx \to  \mathbf{g} \cdot \bfx \label{GL(n)-action}
\end{equation}
with $\mathbf{g} \in GL(n,\R)$.
This transformation induces an analogous action by $\mathbf{g}$ on the tangent and cotangent spaces $T_{\mathbf{0}}\R^n$ and $T^*_{\mathbf{0}}\R^n$,
and hence on the tensors
\begin{align*}
 R & = R_{ijk\ell} (dx^i \wedge dx^j) \circ (dx^k \wedge dx^\ell), \\
H & = \e_{\alpha} \otimes H^\alpha_{ij} dx^i\circ dx^j, \\
A & = A^{ijk} \left(\frac{\partial}{\partial x^i} \circ \frac{\partial}{\partial x^j}\circ \frac{\partial}{\partial x^k} \right),
\end{align*}
and their covariant derivatives, while preserving the normal coordinates condition $\Gamma^k_{ij}(\mathbf{0}) = 0$.

\subsection{Normal form for $n=2$}

When $n=2$, the subspace $\II_\mathbf{0} \subset \mathcal{S}_2$ is spanned by the matrix:
\[
H^3 =
\begin{bmatrix}  H^3_{11} & H^3_{12} \\[0.1in] H^3_{12} & H^3_{22}
\end{bmatrix}.
\]
The nondegeneracy of the embedding $\bfy_0:\Omega \to \R^N$ implies that the matrix $H^3$ is nonzero. Then, by an action of the form \eqref{GL(n)-action}, we can arrange that
\begin{equation}
 H^3 = \begin{bmatrix}  K & 0 \\[0.1in] 0 & 1
\end{bmatrix}, \label{H-normal-form-2D}
\end{equation}
where $K$ is the Gauss curvature of $(M,g)$ at $\bfx = \mathbf{0}$.
The annihilator equations \eqref{annihilator-eqns} then imply that we can choose
\begin{equation}
A^1 = \begin{bmatrix} 0 & 1 \\[0.1in] 1 & 0 \end{bmatrix}, \qquad
A^2 = \begin{bmatrix} 1 & 0 \\[0.1in] 0 & -K \end{bmatrix}.  \label{A-normal-form-2D}
\end{equation}

\subsection{Normal form for $n=3$}

When $n=3$, $\II_\mathbf{0} \subset \mathcal{S}_3$ is the subspace
\[ \II_\mathbf{0} = \text{span}(H^4, H^5, H^6), \]
where, for $\alpha=4,5,6$,
\[ 
H^\alpha = \begin{bmatrix} H^\alpha_{11} & H^\alpha_{12} & H^\alpha_{31} \\[0.1in] H^\alpha_{12} & H^\alpha_{22} 
& H^\alpha_{23} \\[0.1in] H^\alpha_{31} & H^\alpha_{22} & H^\alpha_{33} \end{bmatrix}.
 \]
Each symmetric matrix $H^\alpha$ may also be regarded as representing the quadratic form $H^\alpha_{ij} dx^i  dx^j$  $\in S^2(T^*_{\mathbf{0}}\R^3)$, 
or equivalently, the quadratic polynomial $H^\alpha_{ij} X^i X^j$.

Following \cite{BGY83}, we say that $\II_\mathbf{0}$ is {\em general} if there exists a nonsingular cubic polynomial
$Y = Y_{ijk} X^i X^j X^k$
such that
\[ 
\II_{\mathbf{0}} = \text{span}\left( \frac{\partial Y}{\partial X^1},  \frac{\partial Y}{\partial X^2}, \frac{\partial Y}{\partial X^3} \right). 
\]
In particular, $Y$ must depend on all three variables $X^1, X^2, X^3$.

The following classical lemma may be found, e.g., in \cite{Nowak00}:
\begin{lemma}
If $Y \in S^3(T^*_{\mathbf{0}}\R^3)$ is a nonsingular, homogeneous cubic polynomial, 
then there exists a unique real number $\sigma \neq -\tfrac{1}{2}$ and a basis $(X^1, X^2, X^3)$ of $T^*_{\mathbf{0}}\R^3$ such that
\[ Y = (X^1)^3 + (X^2)^3 + (X^3)^3 + 6\sigma X^1 X^2 X^3. \]
\end{lemma}

It follows that, if $\II_\mathbf{0}$ is general, then, by an action of the form \eqref{GL(n)-action}, we can arrange that
\begin{equation}
\II_{\mathbf{0}} = \text{span}\left(
\begin{bmatrix} 1 & 0 & 0 \\[0.05in] 0 & 0 & \sigma \\[0.05in] 0 & \sigma  & 0 \end{bmatrix},
\begin{bmatrix} 0 & 0 & \sigma \\[0.05in] 0 & 1 & 0 \\[0.05in] \sigma & 0 & 0 \end{bmatrix},
\begin{bmatrix} 0 & \sigma & 0 \\[0.05in] \sigma & 0 & 0 \\[0.05in] 0 & 0  & 1 \end{bmatrix}
\right).   \label{H-normal-form-3D}
\end{equation}
The annihilator equations \eqref{annihilator-eqns} then imply that we can choose
\begin{equation}
A^1 = \begin{bmatrix} -2\sigma & 0 & 0 \\[0.05in] 0 & 0 & 1\\[0.05in] 0 & 1  & 0 \end{bmatrix}, \qquad
A^2 = \begin{bmatrix} 0 & 0 & 1 \\[0.05in] 0 & -2\sigma & 0 \\[0.05in] 1 & 0 & 0 \end{bmatrix}, \qquad
A^3 = \begin{bmatrix} 0 & 1 & 0 \\[0.05in] 1 & 0 & 0 \\[0.05in] 0 & 0  & -2\sigma \end{bmatrix}
. \label{A-normal-form-3D}
\end{equation}

Meanwhile, the Riemann curvature tensor $R$ may be regarded as a quadratic form on the space $\Lambda^2(T_\mathbf{0}\R^3)$;
as such it is represented by the symmetric matrix:
\begin{equation}
\hat{R} = \begin{bmatrix} R_{2323} & R_{2331} & R_{1223} \\[0.05in] R_{2331} & R_{3131} & R_{3112} \\[0.05in] 
R_{1223} & R_{3112} & R_{1212} \end{bmatrix}. \label{define-R}
\end{equation}
The only invariant of $\hat{R}$ under the action \eqref{GL(n)-action} is its signature $(p,q)$.  
The following proposition is a direct consequence of Theorem F in \cite{BGY83}; we will give an independent proof below.

\begin{proposition}
If $\hat{R}$ is nonzero, then the Gauss equations \eqref{Gauss-eqns} have a solution $(H^4$, $H^5$, $H^6)$ 
whose span $\II_{\mathbf{0}}$ is equivalent under the action \eqref{GL(n)-action} to the normal form \eqref{H-normal-form-3D} for some $\sigma$ with $0 < |\sigma| < \tfrac{1}{2}$.
In fact, $\sigma$ may be chosen arbitrarily within this range, the only restrictions being that{\rm :}
\begin{itemize}
\item If the signature of $\hat{R}$ is $(1,0)$, then we must have $\sigma < 0${\rm ;}
\item If the signature of $\hat{R}$ is $(0,1)$, then we must have $\sigma > 0$.
\end{itemize}
\end{proposition}

\begin{proof}
Let $\bar{H}^4, \bar{H}^5, \bar{H}^6$ denote the basis
\begin{equation*}
 \bar{H}^4 = \begin{bmatrix} 1 & 0 & 0 \\[0.05in] 0 & 0 & \sigma \\[0.05in] 0 & \sigma  & 0 \end{bmatrix}, \qquad
\bar{H}^5 = \begin{bmatrix} 0 & 0 & \sigma \\[0.05in] 0 & 1 & 0 \\[0.05in] \sigma & 0 & 0 \end{bmatrix}, \qquad
\bar{H}^6 = \begin{bmatrix} 0 & \sigma & 0 \\[0.05in] \sigma & 0 & 0 \\[0.05in] 0 & 0  & 1 \end{bmatrix}
\end{equation*}
for $\II_\mathbf{0}$.  Then, for $\alpha=4,5,6$, let
\begin{equation*}
 H^\alpha = \gamma^\alpha_\beta \bar{H}^\beta
\end{equation*}
for some invertible matrix $[\gamma^\alpha_\beta]$.  Now, for $\beta=4,5,6$, let $\gamma_\beta$ denote the vector
\[ \gamma_\beta = \begin{bmatrix} \gamma^4_\beta \\[0.05in] \gamma^5_\beta \\[0.05in] \gamma^6_\beta \end{bmatrix}. \]
Then it follows from the Gauss equations \eqref{Gauss-eqns} that the corresponding matrix $\hat{R}$ is given by
\begin{equation} \label{R-matrix-example-1}
 \hat{R} = \begin{bmatrix}
(\gamma_5 \! \cdot \! \gamma_6) - \sigma^2 (\gamma_4 \! \cdot \! \gamma_4) \ & \
\sigma^2 (\gamma_4 \! \cdot \! \gamma_5) - \sigma (\gamma_6 \! \cdot \! \gamma_6) \ & \
\sigma^2 (\gamma_6 \! \cdot \! \gamma_4) - \sigma (\gamma_5 \! \cdot \! \gamma_5) \\[0.05in]
\sigma^2 (\gamma_4 \! \cdot \! \gamma_5) - \sigma (\gamma_6 \! \cdot \! \gamma_6) \ & 
    \ (\gamma_6 \! \cdot \! \gamma_4) - \sigma^2 (\gamma_5 \! \cdot \! \gamma_5) \ & \
\sigma^2 (\gamma_5 \! \cdot \! \gamma_6) - \sigma (\gamma_4 \! \cdot \! \gamma_4) \\[0.05in]
\sigma^2 (\gamma_6 \! \cdot \! \gamma_4) - \sigma (\gamma_5 \! \cdot \! \gamma_5) \ & \
\sigma^2 (\gamma_5 \! \cdot \! \gamma_6) - \sigma (\gamma_4 \! \cdot \! \gamma_4) \ & \
(\gamma_4 \! \cdot \! \gamma_5) - \sigma^2 (\gamma_6 \! \cdot \! \gamma_6)
\end{bmatrix} .
\end{equation}

It suffices to show by example that, with the sign restrictions given above, the vectors $\gamma_4, \gamma_5, \gamma_6$ may 
be chosen so as to obtain a matrix $\hat{R}$ of arbitrary nonzero signature.  
We may achieve this as follows: Let $\gamma_4, \gamma_5, \gamma_6$ be linearly independent unit vectors in $\R^3$, 
oriented so that the angle between any pair of these vectors is equal to the same real number $\theta$.  
Geometric constraints require that $0 < \theta < \frac{2\pi}{3}$, and hence $-\tfrac{1}{2} < \cos \theta < 1$.  
Denote $\cos \theta$ by $\phi$; then from \eqref{R-matrix-example-1}, we have
\begin{equation}\label{R-matrix-example-2}
\hat{R} = \begin{bmatrix}
\phi - \sigma^2 & \sigma^2 \phi - \sigma & \sigma^2 \phi  - \sigma \\[0.05in]
\sigma^2 \phi  - \sigma & \phi - \sigma^2 & \sigma^2 \phi  - \sigma \\[0.05in]
\sigma^2 \phi  - \sigma & \sigma^2 \phi  - \sigma & \phi - \sigma^2
\end{bmatrix}.
\end{equation}
The eigenvalues of the matrix \eqref{R-matrix-example-2} are
\[ \lambda = \phi(1 + 2\sigma^2) - \sigma(\sigma + 2), \ (1-\sigma)(\phi + \sigma + \sigma \phi), \ (1-\sigma)(\phi + \sigma + \sigma \phi). \]
Therefore, for $0 < \sigma < \frac{1}{2}$, we have
\[
\text{sgn}(\hat{R}) = \begin{cases}
(0,3), \ \  & -\frac{1}{2} < \phi < -\frac{\sigma}{\sigma+1}, \\
(0,1), & \phi = -\frac{\sigma}{\sigma+1}, \\
(2,1), & -\frac{\sigma}{\sigma+1} < \phi < \frac{\sigma(\sigma+2)}{1 + 2\sigma^2}, \\
(2,0), & \phi = \frac{\sigma(\sigma+2)}{1 + 2\sigma^2}, \\
(3,0), & \frac{\sigma(\sigma+2)}{1 + 2\sigma^2} < \phi < 1,
\end{cases}
\]
and for $-\frac{1}{2} < \sigma < 0$, we have
\[
\text{sgn}(\hat{R}) = \begin{cases}
(0,3), \ \  & -\frac{1}{2} < \phi < \frac{\sigma(\sigma+2)}{1 + 2\sigma^2} , \\
(0,2), & \phi = \frac{\sigma(\sigma+2)}{1 + 2\sigma^2}, \\
(1,2), & \frac{\sigma(\sigma+2)}{1 + 2\sigma^2} < \phi < -\frac{\sigma}{\sigma+1}, \\
(1,0), & \phi = -\frac{\sigma}{\sigma+1}, \\
(3,0), & -\frac{\sigma}{\sigma+1} < \phi < 1.
\end{cases}
\]
A slight perturbation of the vectors $\gamma_4, \gamma_5, \gamma_6$ will replace the double eigenvalue of $\hat{R}$ 
with distinct eigenvalues, which will lead to $\hat{R}$ attaining the remaining possible 
signatures ($(0,2)$, $(1,2)$, and $(1,1)$ for $\sigma>0$ and $(1,1)$, $(2,1)$, and $(2,0)$ for $\sigma<0$) as $\phi$ varies.
\end{proof}

\begin{remark}
It is possible to show that, when $\hat{R} = 0$, all nondegenerate solutions $(H^4$, $H^5$, $H^6)$ to 
the Gauss equations \eqref{Gauss-eqns} are simultaneously diagonalizable under the action \eqref{GL(n)-action} 
and are therefore equivalent to the normal form \eqref{H-normal-form-3D} with $\sigma=0$.  
The cubic form $A$ thus becomes reducible, with the result that the rank of equations \eqref{d-annihilator-eqns} 
with respect to the variables $h^\alpha_{ijk}$ drops from 27 to 21.  
This drop in rank is the main obstruction to carrying out our construction when $\hat{R}=0$.
\end{remark}

\section{Strong symmetric positivity for the system \eqref{equation:unsymmetrized-system}}\label{ssp-sec-2}

In this section we will prove the following theorem, thereby completing the proof of Theorem \ref{main-theorem-3D}.

\begin{theorem}\label{main-ssp-theorem}
Suppose that either $n=2$ and $K \neq 0$, or $n=3$ and $\hat{R} \neq 0$.
Then the linearized isometric embedding system \eqref{equation:unsymmetrized-system} can be transformed
to a strongly symmetric positive system in a neighborhood of $\bfx = \mathbf{0}$ via a change of variables of the form{\rm :}
\begin{equation}
 x^i = \bar{x}^i + \tfrac{1}{2} c^i_{jk} \bar{x}^j \bar{x}^k, \qquad \bar{\bfv} = (I + \bar{x}^i S_i) \bar{\bfw} , 
 \label{change-of-variables-form}
\end{equation}
where $c^i_{jk} = c^i_{kj} \in \R$ and $S_1, \ldots S_n$ are constant $n \times n$ matrices.
\end{theorem}

In order to prove Theorem \ref{main-ssp-theorem}, we will show that, when $n=2$ or $n=3$, 
for any given real numbers $R_{ijk\ell}$ and $r_{ijk\ell, m}$ satisfying the necessary symmetries 
with $R_{ijk\ell}$ not all equal to zero, there exist real 
numbers $H^\alpha_{ij}, h^\alpha_{ijk}, A^{kij}, a^{kij}_\ell$ that satisfy 
equations \eqref{Gauss-eqns}--\eqref{d-annihilator-eqns}, 
together with a change of variables of the form \eqref {change-of-variables-form}, 
that renders the system \eqref{equation:unsymmetrized-system} strongly symmetric positive at $\bfx=\mathbf{0}$.

At first glance, the strong symmetric positivity condition might appear impossible to achieve: 
From the expressions \eqref{define-matrices} and the normal coordinates condition $\Gamma^k_{ij}(\mathbf{0}) = 0$, 
we have $B(\mathbf{0}) = 0$.  Therefore, symmetric positivity for the system \eqref{equation:unsymmetrized-system} would require that the matrix
\[ \bar{Q}_0 = - \sum_{i=1}^n a^i_i   \]
be positive definite, while strong symmetric positivity would require that each of the diagonal
sub-blocks $(\bar{Q}_1)_{ii} = 2 a^i_i$ (no sum on $i$) of $\bar{Q}_1$ (cf. equation \eqref{define-Q1-matrix}) be positive definite.
Clearly, these two conditions are mutually exclusive, and the situation appears hopeless.  
However, it turns out that a change of variables provides some unexpected flexibility:

\begin{lemma}\label{Q0-Q1-transformation-lemma}
Under the change of variables \eqref{change-of-variables-form}, 
the symmetric linear system \eqref{equation:unsymmetrized-system} with associated quadratic 
forms $\bar{Q}_0$ and $\bar{Q}_1$ at $\bfx = \mathbf{0}$ is transformed to a symmetric system
\begin{equation}
\tilde{A}^i \partial_i \bar{\bfw} + \tilde{B}\bar{\bfw} = \tilde{\bfh}, \label{transformed-system}
\end{equation}
with associated quadratic form $\tilde{\bar{Q}}_0$ at $\bfx=\mathbf{0}$ given by
\begin{equation}
\tilde{\bar{Q}}_0 = - a^i_i + c^i_{ij} A^j, \label{new-Q0}
\end{equation}
and the $(i,j)$th block of $\tilde{\bar{Q}}_1$ {\rm (}cf. equation \eqref{define-Q1-matrix}{\rm )} at $\bfx=\mathbf{0}$ given by
\begin{equation}
\begin{aligned}
(\tilde{\bar{Q}}_1)_{ij} & = \partial_i \tilde{A}^j(\mathbf{0}) + \partial_j \tilde{A}^i(\mathbf{0}) \\
& = a^i_j + a^j_i - (c^i_{jk} + c^j_{ik}) A^k + S_i^{\sf T} A^j + A^j S_i + S_j^{\sf T} A^i + A^i S_j .
\end{aligned} \label{new-Q1-blocks}
\end{equation}
\end{lemma}

\begin{proof}
According to the chain rule, up to first order at $\bfx=\mathbf{0}$, we have
\[ 
\frac{\partial}{\partial \bar{x}^i} = \frac{\partial}{\partial x^i} + c^j_{ik} \bar{x}^k \frac{\partial}{\partial x^j}, \qquad
\frac{\partial}{\partial x^i} = \frac{\partial}{\partial \bar{x}^i} - c^j_{ik} \bar{x}^k \frac{\partial}{\partial \bar{x}^j}.
\]
Therefore, at $\bfx = \mathbf{0}$, we have
\begin{equation}\label{dv-to-dw}
\begin{aligned}
\frac{\partial}{\partial x^i}\bar{\bfv} 
& = \left( \frac{\partial}{\partial \bar{x}^i} - c^j_{ik} \bar{x}^k \frac{\partial}{\partial \bar{x}^j} \right) \left((I + \bar{x}^\ell S_\ell) \bar{\bfw} \right) \\
& = (I + \bar{x}^\ell S_\ell) \left( \frac{\partial}{\partial \bar{x}^i}\bar{\bfw} - c^j_{ik} \bar{x}^k 
\frac{\partial}{\partial \bar{x}^j}\bar{\bfw} \right) + (S_i - c^j_{ik}\bar{x}^k S_j) \bar{\bfw}.
\end{aligned}
\end{equation}
Substitution of \eqref{dv-to-dw} and \eqref{change-of-variables-form} into the linear system \eqref{equation:unsymmetrized-system} yields
\[ 
A^i \left( (I + \bar{x}^\ell S_\ell) \left( \frac{\partial}{\partial \bar{x}^i}\bar{\bfw} 
- c^j_{ik} \bar{x}^k \frac{\partial}{\partial \bar{x}^j}\bar{\bfw} \right) \right) 
+ \left( B (I + \bar{x}^\ell S_\ell) + A^i (S_i - c^j_{ik}\bar{x}^k S_j) \right) \bar{\bfw} = \bfh.
\]
Multiply on the left by $(I + \bar{x}^\ell S_\ell)^{\sf T}$, collect the terms and then relabel them to obtain the system \eqref{transformed-system}, where
\begin{equation} \label{transformed-matrices}
\begin{aligned}
\tilde{A}^i & = (I + \bar{x}^\ell S_\ell)^{\sf T} \left( A^i - c^i_{jk} \bar{x}^k A^j \right) (I + \bar{x}^\ell S_\ell), \\
\tilde{B} & = (I + \bar{x}^\ell S_\ell)^{\sf T}\left( B (I + \bar{x}^\ell S_\ell) + A^i (S_i - c^j_{ik}\bar{x}^k S_j) \right), \\
\tilde{\bfh} & = (I + \bar{x}^\ell S_\ell)^{\sf T}\bfh,
\end{aligned}
\end{equation}
and $\partial_i$ now represents $\frac{\partial}{\partial \bar{x}^i}$.  Finally, computation of
\[ 
\tilde{Q}_0 = \tilde{B} + \tilde{B}^{\sf T} -\sum_{i=1}^n \partial_i \tilde{A}^i, \qquad (\tilde{Q}_1)_{ij} = \partial_i \tilde{A}^j + \partial_j \tilde{A}^i, 
\]
and evaluating at $\bfx = \mathbf{0}$ yields equations \eqref{new-Q0} and \eqref{new-Q1-blocks}.
\end{proof}

In light of Lemma \ref{Q0-Q1-transformation-lemma}, our strategy for proving Theorem \ref{main-ssp-theorem} will be as follows:
\begin{enumerate}
\item[1.] By applying the $GL(n,\R)$ action \eqref{GL(n)-action}, we may assume that $A^{kij}$ and $H^\alpha_{ij}$ 
are as in equations \eqref{H-normal-form-2D}--\eqref{A-normal-form-2D} when $n=2$ and as in equations \eqref{H-normal-form-3D}--\eqref{A-normal-form-3D} when $n=3$.

\item[2.] Identify the values of $a^{kij}_\ell$ for which the system \eqref{equation:unsymmetrized-system} can be transformed to a strongly symmetric 
positive system \eqref{transformed-system} via a change of variables of the form \eqref{change-of-variables-form}.

\item[3.] Identify the values of $h^\alpha_{ijk}$ that satisfy equations \eqref{d-annihilator-eqns} for some $a^{kij}_\ell$ from Step 2.

\item[4.] Show that all possible values of $r_{ijk\ell,m}$ satisfy equations \eqref{d-Gauss-eqns} for some $h^\alpha_{ijk}$ from Step 3.

\item[5.] Conclude that, for any $R_{ijk\ell}$ not all equal to zero and any $r_{ijk\ell,m}$,
there exist $H^\alpha_{ij}$, $A^{kij}$, $h^\alpha_{ijk}$, and $a^{kij}_\ell$ that satisfy equations \eqref{Gauss-eqns}--\eqref{d-annihilator-eqns} 
and for which the system \eqref{equation:unsymmetrized-system} can be transformed to a strongly symmetric positive system \eqref{transformed-system}.
\end{enumerate}

\noindent {\em Proof of Theorem {\rm \ref{main-ssp-theorem}}.}

First we give the proof for the case $n=2$.  We begin by identifying the values of $a^{kij}_\ell$ for which we can arrange that
\begin{equation}
 \tilde{\bar{Q}}_0 = \lambda I_2, \qquad \tilde{\bar{Q}}_1 =  \mu I_4, \label{optimal-Qs-2d}
\end{equation}
for given real numbers $\lambda, \mu >0$, where $I_2$ and $I_4$ denote the $2\times 2$ and $4\times 4$ identity matrices, respectively.

By applying the $GL(2)$ action \eqref{GL(n)-action}, we can assume that
\[ A^1 = \begin{bmatrix} 0 & 1 \\[0.1in] 1 & 0 \end{bmatrix}, \qquad
A^2 = \begin{bmatrix} 1 & 0 \\[0.1in] 0 & -K \end{bmatrix}, \qquad
H^3 = \begin{bmatrix}  K & 0 \\[0.1in] 0 & 1
\end{bmatrix}. \]
Set
\[ D = \begin{bmatrix} 0 & 0 \\[0.1in] 0 & 1 \end{bmatrix}, \]
so that the matrices $A^1, A^2$, and $D$ form a basis for $\mathcal{S}_2$,
and write the matrices $S_1$ and $S_2$ as
\[ S_1 = \begin{bmatrix} s^{11}_1 & s^{12}_1 \\[0.1in] s^{21}_1 & s^{22}_1 \end{bmatrix}, \qquad S_2 
= \begin{bmatrix} s^{11}_2 & s^{12}_2 \\[0.1in] s^{21}_2 & s^{22}_2 \end{bmatrix}. \]
Then, after some computation, equation \eqref{new-Q0} can be written as
\begin{equation}\label{new-Q0-details-2D}
\tilde{\bar{Q}}_0 = -(a^1_1 + a^2_2) + (c^1_{11} + c^2_{12}) A^1 + (c^1_{12} + c^2_{22}) A^2,
\end{equation}
and the equations \eqref{new-Q1-blocks} can be written as
\begin{equation} \label{new-Qij-details-2D}
\begin{aligned}
(\tilde{\bar{Q}}_1)_{11} & = 2 a^1_1 - 2(c^1_{11} - s^{11}_1 - s^{22}_1) A^1 - 2(c^1_{12} + Z_{112}(s)) A^2 + 4(s^{12}_1 + K s^{21}_1) D ,\\
(\tilde{\bar{Q}}_1)_{22} & = 2 a^2_2 - 2(c^2_{12} + Z_{221}(s)) A^1 - 2(c^2_{22} - 2s^{11}_2) A^2 + 4K(s^{11}_2 - s^{22}_2) D, \\
(\tilde{\bar{Q}}_1)_{12} & = a^1_2 + a^2_1 - (c^1_{12} + c^2_{11} + Z_{121}(s)) A^1 - (c^1_{22} + c^2_{12} + Z_{122}(s)) A^2 \\
& \qquad   + 2(s^{12}_2 + K(s^{11}_1 + s^{21}_2 - s^{22}_1)) D,
\end{aligned}
\end{equation}
where $Z_{ijk}(s)$ represents a linear combination of the $s^{ij}_k$ whose precise form is irrelevant.  
Regardless of the values of $a^{kij}_\ell$, we can set
\[ 
(\tilde{\bar{Q}}_1)_{ij} = \delta_{ij} \mu I_2 
\]
and solve equations \eqref{new-Qij-details-2D} for the variables $s^{12}_1, s^{22}_2, s^{12}_2$ (from the coefficients of $D$)
and $s^{22}_1,  s^{11}_2, c^1_{12}, c^2_{12}, c^2_{11}, c^1_{22}$ (from the coefficients of $A^1$ and $A^2$).  
Note that this solution makes use of the assumption that $K \neq 0$.
Then we can set
\[ 
\tilde{\bar{Q}}_0 = \lambda I_2 
\]
and solve equations \eqref{new-Q0-details-2D} for the variables $c^1_{11}$ and $c^2_{22}$ if and only if the matrix
\[ 
a^1_1 + a^2_2 + \lambda I_2  = \begin{bmatrix} a^{111}_1 + a^{112}_2 + \lambda & a^{112}_1 + a^{122}_2 \\[0.1in] 
a^{112}_1 + a^{122}_2 & a^{122}_1 + a^{222}_2 + \lambda \end{bmatrix}
 \]
is a linear combination of $A^1$ and $A^2$, which in turn is true if and only if
\begin{equation}\label{aijkl-condition-2D}
 (a^{122}_1 + a^{222}_2 + \lambda) + K(a^{111}_1 + a^{112}_2 + \lambda) = 0.
\end{equation}
Thus, the strong symmetric positivity condition \eqref{optimal-Qs-2d} can be realized 
if and only if the $a^{kij}_\ell$ satisfy equation \eqref{aijkl-condition-2D}.

The next step is to identify the values of $h^3_{ijk}$ that satisfy equations \eqref{d-annihilator-eqns} for some $a^{kij}_\ell$ 
satisfying equation \eqref{aijkl-condition-2D}.  Equations \eqref{d-annihilator-eqns} may be written in matrix form as
\begin{equation}
 \langle A^k, h^3_\ell \rangle + \langle H^3, a^k_\ell \rangle = 0. \label{d-annihilator-eqns-matrixform-2D}
\end{equation}
The condition \eqref{aijkl-condition-2D} is equivalent to
\[ \langle H^3, a^1_1 + a^2_2 \rangle =  -(K+1)\lambda; \]
therefore, \eqref{d-annihilator-eqns-matrixform-2D} implies that we must have
\begin{equation*}
 \langle A^1, h^3_1 \rangle + \langle A^2, h^3_2 \rangle = -\langle H^3, a^1_1 + a^2_2 \rangle = (K+1)\lambda,
\end{equation*}
or, equivalently,
\begin{equation}
3 h^3_{112} - K h^3_{222} = (K+1)\lambda.  \label{h-condition-2D}
\end{equation}
Conversely, for any values of $h^3_{ijk}$ that satisfy the condition \eqref{h-condition-2D}, 
there exist values of $a^{kij}_\ell$ that satisfy the condition \eqref{aijkl-condition-2D}.

Finally, consider equations \eqref{d-Gauss-eqns}, which can be written as
\begin{equation}\label{d-Gauss-eqns-details-2D}
\begin{aligned}
K h^3_{122} + h^3_{111} & = k_1, \\
K h^3_{222} + h^3_{112} & = k_2.
\end{aligned}
\end{equation}
The values of $h^3_{ijk}$ may be chosen arbitrarily, subject only to the condition \eqref{h-condition-2D}; 
therefore, any given values of $k_1$ and $k_2$ may be realized by an appropriate choice of $h^3_{ijk}$.

We conclude that, for any $K \neq 0$ and any $k_1, k_2$, there exist solutions $h^3_{ijk}$ and $a^{kij}_\ell$ to 
equations \eqref{Gauss-eqns}--\eqref{d-annihilator-eqns} that satisfy the conditions \eqref{aijkl-condition-2D} and \eqref{h-condition-2D}, 
and hence the linearized system \eqref{equation:unsymmetrized-system} can be transformed to a strongly symmetric positive system via a change of variables 
of the form \eqref{change-of-variables-form}.  This completes the proof for $n=2$.

Now consider the case $n=3$.  The argument is essentially the same as for $n=2$, but the linear algebra requires a bit more effort.  
We begin by identifying the values of $a^{kij}_\ell$ for which we can arrange that
\begin{equation}
 \tilde{\bar{Q}}_0 = \lambda I_3, \qquad \tilde{\bar{Q}}_1 =  \mu I_9, \label{optimal-Qs-3d}
\end{equation}
for given real numbers $\lambda, \mu >0$, where $I_3$ and $I_9$ denote the $3\times 3$ and $9\times 9$ identity matrices, respectively.

By applying the $GL(3)$ action \eqref{GL(n)-action}, we can assume that
\[ A^1 = \begin{bmatrix} -2\sigma & 0 & 0 \\[0.05in] 0 & 0 & 1\\[0.05in] 0 & 1  & 0 \end{bmatrix}, \qquad
A^2 = \begin{bmatrix} 0 & 0 & 1 \\[0.05in] 0 & -2\sigma & 0 \\[0.05in] 1 & 0 & 0 \end{bmatrix}, \qquad
A^3 = \begin{bmatrix} 0 & 1 & 0 \\[0.05in] 1 & 0 & 0 \\[0.05in] 0 & 0  & -2\sigma \end{bmatrix} \]
with $0 < |\sigma| < \tfrac{1}{2}$.
Let $\bar{H}^4, \bar{H}^5, \bar{H}^6$ denote the basis
\begin{equation}\label{Hbar-basis}
 \bar{H}^4 = \begin{bmatrix} 1 & 0 & 0 \\[0.05in] 0 & 0 & \sigma \\[0.05in] 0 & \sigma  & 0 \end{bmatrix}, \qquad
\bar{H}^5 = \begin{bmatrix} 0 & 0 & \sigma \\[0.05in] 0 & 1 & 0 \\[0.05in] \sigma & 0 & 0 \end{bmatrix}, \qquad
\bar{H}^6 = \begin{bmatrix} 0 & \sigma & 0 \\[0.05in] \sigma & 0 & 0 \\[0.05in] 0 & 0  & 1 \end{bmatrix}
\end{equation}
for $\II_\mathbf{0}$.  Then, for $\alpha=4,5,6$, we can write
\begin{equation} \label{fancy-H}
 H^\alpha = \gamma^\alpha_\beta \bar{H}^\beta
\end{equation}
for some invertible matrix $[\gamma^\alpha_\beta]$.
Set
\[ 
D_1 = \begin{bmatrix} 1 & 0 & 0 \\[0.1in] 0 & 0 & 0 \\[0.1in] 0 & 0 & 0 \end{bmatrix}, \qquad
D_2 = \begin{bmatrix} 0 & 0 & 0 \\[0.1in] 0 & 1 & 0 \\[0.1in] 0 & 0 & 0 \end{bmatrix}, \qquad
D_3 = \begin{bmatrix} 0 & 0 & 0 \\[0.1in] 0 & 0 & 0 \\[0.1in] 0 & 0 & 1 \end{bmatrix},
\]
so that the matrices $A^1, A^2, A^3, D_1, D_2, D_3$ form a basis for $\mathcal{S}_3$,
and write the matrices $S_1, S_2, S_3$ as
\[ 
S_1 = \begin{bmatrix} s^{11}_1 & s^{12}_1 & s^{13}_1 \\[0.1in] s^{21}_1 & s^{22}_1 & s^{23}_1 \\[0.1in] s^{31}_1 & s^{32}_1 & s^{33}_1 \end{bmatrix}, \qquad
S_2 = \begin{bmatrix} s^{11}_2 & s^{12}_2 & s^{13}_2 \\[0.1in] s^{21}_2 & s^{22}_2 & s^{23}_2 \\[0.1in] s^{31}_2 & s^{32}_2 & s^{33}_2 \end{bmatrix}, \qquad
S_3 = \begin{bmatrix} s^{11}_3 & s^{12}_3 & s^{13}_3 \\[0.1in] s^{21}_3 & s^{22}_3 & s^{23}_3 \\[0.1in] s^{31}_3 & s^{32}_3 & s^{33}_3 \end{bmatrix}. 
\]
Then, after some computation, equation \eqref{new-Q0} can be written as
\begin{equation}\label{new-Q0-details-3D}
\tilde{\bar{Q}}_0 = -(a^1_1 + a^2_2 + a^3_3) + (c^1_{11} + c^2_{12} + c^3_{13}) A^1 + (c^1_{12} + c^2_{22} + c^3_{23}) A^2 + (c^1_{13} + c^2_{23} + c^3_{33}) A^3,
\end{equation}
and the equations \eqref{new-Q1-blocks} can be written as
\begin{equation} \label{new-Qij-details-3D}
\begin{aligned}
(\tilde{\bar{Q}}_1)_{11} & = 2 a^1_1 - 2(c^1_{11} - s^{22}_1 - s^{33}_1) A^1 - 2(c^1_{12} + Z_{112}(s)) A^2 - 2(c^1_{13} + Z_{113}(s)) A^3 \\
& \   + 4\sigma(s^{22}_1 + s^{33}_1 - 2 s^{11}_1) D_1 + 4(s^{32}_1 + \sigma s^{21}_1 - 2\sigma^2 s^{13}_1) D_2 + 4(s^{23}_1 + \sigma s^{31}_1 - 2\sigma^2 s^{12}_1) D_3, \\
(\tilde{\bar{Q}}_1)_{22} & = 2 a^2_2 - 2(c^2_{12} + Z_{221}(s)) A^1 - 2(c^2_{22} - s^{11}_2 - s^{33}_2) A^2 - 2(c^2_{23} + Z_{223}(s)) A^3 \\
& \   + 4(s^{31}_2 + \sigma s^{12}_2 - 2\sigma^2 s^{23}_2) D_1 + 4\sigma(s^{11}_2 + s^{33}_2 - 2 s^{22}_2) D_2 + 4(s^{13}_2 + \sigma s^{32}_2 - 2\sigma^2 s^{21}_2) D_3, \\
(\tilde{\bar{Q}}_1)_{33} & = 2 a^3_3 - 2(c^3_{13} + Z_{331}(s)) A^1 - 2(c^3_{23} + Z_{332}(s)) A^2 - 2(c^3_{33} - s^{11}_3 - s^{22}_3) A^3 \\
& \   + 4(s^{21}_3 + \sigma s^{13}_3 - 2\sigma^2 s^{32}_3) D_1 + 4(s^{12}_3 + \sigma s^{23}_3 - 2\sigma^2 s^{31}_3) D_2 + 4\sigma(s^{11}_3 + s^{22}_3 - 2 s^{33}_3) D_3, \\
(\tilde{\bar{Q}}_1)_{12} & = a^1_2 + a^2_1 - (c^1_{12} + c^2_{11} + Z_{121}(s)) A^1 - (c^1_{22} + c^2_{12} + Z_{122}(s)) A^2 \\
& \  - (c^1_{23} + c^2_{13} + Z_{123}(s)) A^3 + 2\left(s^{31}_1 + \sigma(s^{12}_1 + s^{22}_2 + s^{33}_2 - 2s^{11}_2) - 2\sigma^2 s^{23}_1\right) D_1 \\ 
& \   + 2\left( s^{32}_2 + \sigma(s^{21}_2 + s^{11}_1 + s^{33}_1 - 2s^{22}_1) - 2\sigma^2 s^{13}_2 \right) D_2 \\
& \ \ +
2\left( s^{13}_1 + s^{23}_2 + \sigma( s^{32}_1 + s^{31}_2) - 2\sigma^2(s^{12}_2 + s^{21}_1)\right) D_3, \\
(\tilde{\bar{Q}}_1)_{23} & = a^2_3 + a^3_2 - (c^2_{13} + c^3_{12} + Z_{231}(s)) A^1 - (c^2_{23} + c^3_{22} + Z_{232}(s)) A^2 \\
& \  - (c^2_{33} + c^3_{23} + Z_{233}(s)) A^3 +
2\left( s^{21}_2 + s^{31}_3 + \sigma( s^{13}_2 + s^{12}_3) - 2\sigma^2(s^{32}_2 + s^{23}_3)\right) D_1 \\
& \ + 2\left( s^{12}_2 + \sigma(s^{23}_2 + s^{11}_3 + s^{33}_3 - 2s^{22}_3) - 2\sigma^2 s^{31}_2 \right) D_2 \\
& \ \ + 2\left(s^{13}_3 + \sigma(s^{11}_2 + s^{22}_2 + s^{32}_3 - 2s^{33}_2) - 2\sigma^2 s^{21}_3\right) D_3, \\
(\tilde{\bar{Q}}_1)_{31} & = a^3_1 + a^1_3 - (c^3_{11} + c^1_{13} + Z_{311}(s)) A^1 - (c^3_{12} + c^1_{23} + Z_{312}(s)) A^2 \\
& \  - (c^3_{13} + c^1_{33} + Z_{313}(s)) A^3 + 2\left( s^{21}_1 + \sigma(s^{13}_1 + s^{22}_3 + s^{33}_3 - 2s^{11}_3) - 2\sigma^2 s^{32}_1 \right) D_1 \\
& \ + 2\left( s^{12}_1 + s^{32}_3 + \sigma( s^{23}_1 + s^{21}_3) - 2\sigma^2(s^{31}_1 + s^{13}_3)\right) D_2 \\
& \ \ + 2\left(s^{23}_3 + \sigma(s^{11}_1 + s^{22}_1 + s^{31}_3 - 2s^{33}_1) - 2\sigma^2 s^{12}_3\right) D_3,
\end{aligned}
\end{equation}
where $Z_{ijk}(s)$ represents a linear combination of the $s^{ij}_k$ whose precise form is irrelevant.  Regardless of the values of $a^{kij}_\ell$, we can set
\[ (\tilde{\bar{Q}}_{ij}) = \delta_{ij} \mu I_3 \]
and solve equations \eqref{new-Qij-details-3D} for the variables
\[ 
s_1^{1,1}, s_1^{2,3}, s_1^{3,2}, s_2^{2,2}, s_2^{1,3}, s_2^{3,1}, s_3^{3,3}, s_3^{1,2}, s_3^{2,1}, s_1^{3,1}, s_2^{3,2}, 
s_2^{2,3}, s_2^{1,2}, s_3^{1,3}, s_3^{3,1}, s_1^{2,1}, s_3^{2,3}, s_3^{3,2} 
\]
(from the coefficients of $D_1, D_2, D_3$) and
\[ 
s_1^{3,3}, s_2^{1,1}, s_3^{2,2}, c^1_{1,2}, c^1_{1,3}, c^2_{1,2}, c^2_{2,3}, c^3_{1,3}, c^3_{2,3}, c^2_{1,1},
 c^1_{2,2}, c^2_{3,3}, c^3_{2,2}, c^1_{3,3}, c^3_{1,1}, c^1_{2,3}, c^2_{1,3}, c^3_{1,2} 
 \]
(from the coefficients of $A^1, A^2, A^3$).  This solution makes use of the fact that $0 < |\sigma| < \tfrac{1}{2}$,
and while the explicit solution is rather complicated, it should be fairly clear that such a solution exists for $|\sigma|>0$ sufficiently small.  
Then we can set
\[ 
\tilde{\bar{Q}}_0 = \lambda I_3 
\]
and solve equations \eqref{new-Q0-details-3D} for the variables $c^1_{11}, c^2_{22}, c^3_{33}$ if and only if the matrix
\[ 
a^1_1 + a^2_2 + a^3_3 + \lambda I_3  = \begin{bmatrix} a^{111}_1 + a^{112}_2 + a^{113}_3 + \lambda & a^{112}_1 + a^{122}_2 + a^{123}_3 & a^{113}_1 + a^{123}_2 + a^{133}_3 \\[0.1in] 
a^{112}_1 + a^{122}_2 + a^{123}_3 & a^{122}_1 + a^{222}_2 + a^{223}_3 + \lambda & a^{123}_1 + a^{223}_2 + a^{233}_3 \\[0.1in] 
a^{113}_1 + a^{123}_2 + a^{133}_3 & a^{123}_1 + a^{223}_2 + a^{233}_3 & a^{133}_1 + a^{233}_2 + a^{333}_3 + \lambda
\end{bmatrix}
 \]
is a linear combination of $A^1$, $A^2$, and $A^3$, which in turn is true if and only if
\begin{equation}\label{aijkl-condition-3D}
\begin{gathered}
a^{111}_1 + a^{112}_2 + a^{113}_3 + 2\sigma(a^{123}_1 + a^{223}_2 + a^{233}_3) + \lambda = 0, \\
a^{122}_1 + a^{222}_2 + a^{223}_3 + 2\sigma(a^{113}_1 + a^{123}_2 + a^{133}_3) + \lambda = 0, \\
a^{133}_1 + a^{233}_2 + a^{333}_3 + 2\sigma(a^{112}_1 + a^{122}_2 + a^{123}_3) + \lambda = 0.
\end{gathered}
\end{equation}
Thus, the strong symmetric positivity condition \eqref{optimal-Qs-3d} can be realized if and only if the $a^{kij}_\ell$ satisfy equations \eqref{aijkl-condition-3D}.

The next step is to identify the values of $h^\alpha_{ijk}$ that satisfy equations \eqref{d-annihilator-eqns} 
for some $a^{kij}_\ell$ satisfying equation \eqref{aijkl-condition-3D}.  Equations \eqref{d-annihilator-eqns} may be written in matrix form as
\begin{equation}
 \langle A^k, h^\alpha_\ell \rangle + \langle H^\alpha, a^k_\ell \rangle = 0. \label{d-annihilator-eqns-matrixform-3D}
\end{equation}
The conditions \eqref{aijkl-condition-3D} are equivalent to
\[ \langle \bar{H}^\alpha, a^1_1 + a^2_2 + a^3_3 \rangle =  -\lambda, \qquad \alpha=4,5,6; \]
therefore,
\[ \langle H^\alpha, a^1_1 + a^2_2 + a^3_3 \rangle = \langle \gamma^\alpha_\beta \bar{H}^\beta, a^1_1 + a^2_2 + a^3_3 \rangle =
  -(\gamma^\alpha_4 + \gamma^\alpha_5 + \gamma^\alpha_6) \lambda, \qquad \alpha=4,5,6. \]
Then equation \eqref{d-annihilator-eqns-matrixform-3D} implies that we must have
\begin{equation*}
 \langle A^1, h^\alpha_1 \rangle + \langle A^2, h^\alpha_2 \rangle +  \langle A^3, h^\alpha_3 \rangle 
 = -\langle H^\alpha, a^1_1 + a^2_2 + a^3_3 \rangle = (\gamma^\alpha_4 + \gamma^\alpha_5 + \gamma^\alpha_6)\lambda, \qquad \alpha=4,5,6,
\end{equation*}
or, equivalently,
\begin{equation}\label{h-condition-3D}
6 h^\alpha_{123} - 2\sigma(h^\alpha_{111} + h^\alpha_{222} + h^\alpha_{333}) = (\gamma^\alpha_4 + \gamma^\alpha_5 + \gamma^\alpha_6)\lambda, \qquad \alpha=4,5,6.
\end{equation}
Conversely, for any values of $h^\alpha_{ijk}$ that satisfy the conditions \eqref{h-condition-3D}, 
there exist values of $a^{kij}_\ell$ that satisfy the conditions \eqref{aijkl-condition-3D}.

Equations \eqref{d-Gauss-eqns} are considerably more complicated here than in the $n=2$ case.  
Taking the second Bianchi equations into account, there are 15 equations for the 15 components $r_{ijk\ell,m}$, 
with left-hand sides that are linear functions of the 30 components $h^\alpha_{ijk}$.  
We will regard equations \eqref{d-Gauss-eqns} as defining a linear map $\tilde{G}$ from the$30$-dimensional space $\mathcal{H}$ of $h^\alpha_{ijk}$ values 
to the $15$-dimensional space $\mathcal{R}$ of $r_{ijk\ell,m}$ values; 
what remains to show is that the restriction of $\tilde{G}$ to the 27-dimensional affine subspace defined by equations \eqref{h-condition-3D} is surjective onto $\mathcal{R}$.

First, observe that we can write equations \eqref{d-Gauss-eqns} in matrix form as
\begin{equation}\label{d-Gauss-eqns-matrix-form}
\sum_{\alpha = 4}^6 \tilde{G}^\alpha \hat{h}^\alpha = \hat{r},
\end{equation}
where $\tilde{G}^\alpha$ denotes the $15 \times 10$ matrix
\[ \tilde{G}^{\alpha} = \begin{bmatrix}
H^{\alpha}_{22} & 0 & 0 & -2 H^{\alpha}_{12} & 0 & 0 & H^{\alpha}_{11} & 0 & 0 & 0 \\[0.1in]
H^{\alpha}_{33} & 0 & 0 & 0 & -2 H^{\alpha}_{31} & 0 & 0 & H^{\alpha}_{11} & 0 & 0  \\[0.1in]
0 & 0 & 0 & -H^{\alpha}_{33} & H^{\alpha}_{23} & 0 & 0 & -H^{\alpha}_{12} & 0 & H^{\alpha}_{31}  \\[0.1in]
-H^{\alpha}_{23} & 0 & 0 & H^{\alpha}_{31} & H^{\alpha}_{12} & 0 & 0 & 0 & 0 & -H^{\alpha}_{11} \\[0.1in]
0 & 0 & 0 & H^{\alpha}_{23} & -H^{\alpha}_{22} & 0 & -H^{\alpha}_{31} & 0 & 0  & H^{\alpha}_{12}  \\[0.1in]
0 & H^{\alpha}_{33} & 0 & 0 & 0 & -2 H^{\alpha}_{23} & 0 & 0 & H^{\alpha}_{22} & 0 \\[0.1in]
0 & H^{\alpha}_{11} & 0 & H^{\alpha}_{22} & 0 & 0 & -2 H^{\alpha}_{12} & 0 & 0 & 0 \\[0.1in]
0 & 0 & 0 & -H^{\alpha}_{23} & 0 & -H^{\alpha}_{11} & H^{\alpha}_{31} & 0 & 0 & H^{\alpha}_{12}   \\[0.1in]
0 & -H^{\alpha}_{31} & 0 & 0 & 0 & H^{\alpha}_{12} & H^{\alpha}_{23} & 0 & 0 & -H^{\alpha}_{22} \\[0.1in]
0 & 0 & 0 & 0 & 0 & H^{\alpha}_{31} & -H^{\alpha}_{33} & 0 & -H^{\alpha}_{12} & H^{\alpha}_{23}  \\[0.1in]
0 & 0 & H^{\alpha}_{11} & 0 & H^{\alpha}_{33} & 0 & 0 & -2 H^{\alpha}_{31} & 0 & 0 \\[0.1in]
0 & 0 & H^{\alpha}_{22} & 0 & 0 & H^{\alpha}_{33} & 0 & 0 & -2 H^{\alpha}_{23} & 0  \\[0.1in]
0 & 0 & 0 & 0 & 0 & -H^{\alpha}_{31} & 0 & -H^{\alpha}_{22} & H^{\alpha}_{12} & H^{\alpha}_{23} \\[0.1in]
0 & 0 & -H^{\alpha}_{12} & 0 & 0 & 0 & 0 & H^{\alpha}_{23} & H^{\alpha}_{31} & -H^{\alpha}_{33}  \\[0.1in]
0 & 0 & 0 & 0 & -H^{\alpha}_{23} & 0 & 0 & H^{\alpha}_{12} & -H^{\alpha}_{11} & H^{\alpha}_{31}
\end{bmatrix}; \]
$\hat{h}^\alpha$ denotes the vector
\[ 
\hat{h}^{\alpha} = 
\begin{bmatrix} h^{\alpha}_{111} & h^{\alpha}_{222} & h^{\alpha}_{333} & h^{\alpha}_{112} & h^{\alpha}_{311} 
& h^{\alpha}_{223} & h^{\alpha}_{122} & h^{\alpha}_{331} & h^{\alpha}_{233} & h^{\alpha}_{123}   
\end{bmatrix}^{\sf T} ,  
\]
and $\hat{r}$ denotes the vector
\begin{align*}
 \hat{r} & = [
r_{1212,1}\  \ r_{3131,1} \  \  r_{2331,1} \  \  r_{3112,1} \  \  r_{1223,1} \  \ r_{2323,2} \  \  r_{1212,2} \  \  r_{3112,2}   \\
& \qquad   r_{1223,2}\ \ r_{2331,2} \  \  r_{3131,3} \  \  r_{2323,3} \  \   r_{1223,3}\ \ r_{2331,3} \  \  r_{3112,3}
]^{\sf T}  .
\end{align*}
Thus the map $\tilde{G}$ is represented by the $15 \times 30$ matrix
\[ \begin{bmatrix} \tilde{G}^4 & \tilde{G}^5 & \tilde{G}^6 \end{bmatrix} \]
acting on the vector
\[ \hat{h} = \begin{bmatrix} \hat{h}^4 \\[0.1in] \hat{h}^5 \\[0.1in] \hat{h}^6 \end{bmatrix}. \]

Now, let $G$ denote the restriction of $\tilde{G}$ to the 27-dimensional subspace defined by equations \eqref{h-condition-3D}.  
By solving equations \eqref{h-condition-3D} for $h^\alpha_{123}$ and substituting into equation \eqref{d-Gauss-eqns-matrix-form}, we can represent $G$ as
\[  
G(\hat{h}) = \sum_{\alpha = 4}^6 G^\alpha \hat{\bar{h}}^\alpha + \hat{r}_0, 
\]
where $G^\alpha$ denotes the $15 \times 9$ matrix
\begin{multline}\label{G-alpha-matrix}
 G^{\alpha} = \\
 \begin{bmatrix}
H^{\alpha}_{22} \! \! & \!  \!  0 \! \! & \!  \!  0 \! \! & \!  \!  -2 H^{\alpha}_{12} \! \! & \!  \!  0 \! \! & \!  \!  0 \! \! 
   & \!  \!  H^{\alpha}_{11} \! \! & \!  \!  0 \! \! & \!  \!  0  \\[0.1in]
H^{\alpha}_{33} \! \! & \!  \!  0 \! \! & \!  \!  0 \! \! & \!  \!  0 \! \! & \!  \!  -2 H^{\alpha}_{31} \! \! & \!  \!  0 \! \! 
  & \!  \!  0 \! \! & \!  \!  H^{\alpha}_{11} \! \! & \!  \!  0   \\[0.1in]
\tfrac{1}{3}\sigma H^{\alpha}_{31} \! \! & \!  \!  \tfrac{1}{3}\sigma H^{\alpha}_{31} \! \! & \!  \!  \tfrac{1}{3}\sigma H^{\alpha}_{31} \! \! 
   & \!  \!  -H^{\alpha}_{33} \! \! & \!  \!  H^{\alpha}_{23} \! \! & \!  \!  0 \! \! & \!  \!  0 \! \! & \!  \!  -H^{\alpha}_{12} \! \! & \!  \!  0   \\[0.1in]
-H^{\alpha}_{23} \! - \! \tfrac{1}{3}\sigma H^{\alpha}_{11} \! \! & \!  \!  - \tfrac{1}{3}\sigma H^{\alpha}_{11} \! \! 
    & \!  \!  - \tfrac{1}{3}\sigma H^{\alpha}_{11} \! \! & \!  \!  H^{\alpha}_{31} \! \! & \!  \!  H^{\alpha}_{12} \! \! & \!  \!  0 \! \! 
      & \!  \!  0 \! \! & \!  \!  0 \! \! & \!  \!  0  \\[0.1in]
\tfrac{1}{3} \sigma H^{\alpha}_{12} \! \! & \!  \!  \tfrac{1}{3} \sigma H^{\alpha}_{12} \! \! 
   & \!  \!  \tfrac{1}{3} \sigma H^{\alpha}_{12} \! \! & \!  \!  H^{\alpha}_{23} \! \! & \!  \!  -H^{\alpha}_{22} \! \! & \!  \!  0 \! \! 
       & \!  \!  -H^{\alpha}_{31} \! \! & \!  \!  0 \! \! & \!  \!  0    \\[0.1in]
0 \! \! & \!  \!  H^{\alpha}_{33} \! \! & \!  \!  0 \! \! & \!  \!  0 \! \! & \!  \!  0 \! \! 
   & \!  \!  -2 H^{\alpha}_{23} \! \! & \!  \!  0 \! \! & \!  \!  0 \! \! & \!  \!  H^{\alpha}_{22}  \\[0.1in]
0 \! \! & \!  \!  H^{\alpha}_{11} \! \! & \!  \!  0 \! \! & \!  \!  H^{\alpha}_{22} \! \! & \!  \!  0 \! \! 
   & \!  \!  0 \! \! & \!  \!  -2 H^{\alpha}_{12} \! \! & \!  \!  0 \! \! & \!  \!  0  \\[0.1in]
\tfrac{1}{3}\sigma H^{\alpha}_{12} \! \! & \!  \!  \tfrac{1}{3}\sigma H^{\alpha}_{12} \! \! 
  & \!  \!  \tfrac{1}{3}\sigma H^{\alpha}_{12} \! \! & \!  \!  -H^{\alpha}_{23} \! \! & \!  \!  0 \! \! 
     & \!  \!  -H^{\alpha}_{11} \! \! & \!  \!  H^{\alpha}_{31} \! \! & \!  \!  0 \! \! & \!  \!  0    \\[0.1in]
-\tfrac{1}{3}\sigma H^{\alpha}_{22} \! \! & \!  \!  -H^{\alpha}_{31} \! -\! \tfrac{1}{3}\sigma H^{\alpha}_{22} \! \! 
   & \!  \!  -\tfrac{1}{3}\sigma H^{\alpha}_{22} \! \! & \!  \!  0 \! \! & \!  \!  0 \! \! & \!  \!  H^{\alpha}_{12} \! \! 
     & \!  \!  H^{\alpha}_{23} \! \! & \!  \!  0 \! \! & \!  \!  0  \\[0.1in]
\tfrac{1}{3}\sigma H^{\alpha}_{23} \! \! & \!  \!  \tfrac{1}{3}\sigma H^{\alpha}_{23} \! \! 
  & \!  \!  \tfrac{1}{3}\sigma H^{\alpha}_{23} \! \! & \!  \!  0 \! \! & \!  \!  0 \! \! & \!  \!  H^{\alpha}_{31} \! \! 
    & \!  \!  -H^{\alpha}_{33} \! \! & \!  \!  0 \! \! & \!  \!  -H^{\alpha}_{12}   \\[0.1in]
0 \! \! & \!  \!  0 \! \! & \!  \!  H^{\alpha}_{11} \! \! & \!  \!  0 \! \! & \!  \!  H^{\alpha}_{33} \! \! & \!  \!  0 \! \! 
  & \!  \!  0 \! \! & \!  \!  -2 H^{\alpha}_{31} \! \! & \!  \!  0  \\[0.1in]
0 \! \! & \!  \!  0 \! \! & \!  \!  H^{\alpha}_{22} \! \! & \!  \!  0 \! \! & \!  \!  0 \! \! & \!  \!  H^{\alpha}_{33} \! \! 
  & \!  \!  0 \! \! & \!  \!  0 \! \! & \!  \!  -2 H^{\alpha}_{23}   \\[0.1in]
\tfrac{1}{3}\sigma H^{\alpha}_{23} \! \! & \!  \!  \tfrac{1}{3}\sigma H^{\alpha}_{23}  \! \! & \!  \!  \tfrac{1}{3}\sigma H^{\alpha}_{23}  \! \! 
  & \!  \!  0 \! \! & \!  \!  0 \! \! & \!  \!  -H^{\alpha}_{31} \! \! & \!  \!  0 \! \! & \!  \!  -H^{\alpha}_{22} \! \! & \!  \!  H^{\alpha}_{12} \\[0.1in]
-\tfrac{1}{3}\sigma H^{\alpha}_{33} \! \! & \!  \!  -\tfrac{1}{3}\sigma H^{\alpha}_{33} \! \! & \!  \!  -H^{\alpha}_{12} \! - \! \tfrac{1}{3}\sigma H^{\alpha}_{33} \! \! 
   & \!  \!  0 \! \! & \!  \!  0 \! \! & \!  \!  0 \! \! & \!  \!  0 \! \! & \!  \!  H^{\alpha}_{23} \! \! & \!  \!  H^{\alpha}_{31}   \\[0.1in]
\tfrac{1}{3}\sigma H^{\alpha}_{31} \! \! & \!  \!  \tfrac{1}{3}\sigma H^{\alpha}_{31} \! \! & \!  \!  \tfrac{1}{3}\sigma H^{\alpha}_{31} \! \! 
  & \!  \!  0 \! \! & \!  \!  -H^{\alpha}_{23} \! \! & \!  \!  0 \! \! & \!  \!  0 \! \! & \!  \!  H^{\alpha}_{12} \! \! & \!  \!  -H^{\alpha}_{11}
\end{bmatrix};
\end{multline}
$\hat{\bar{h}}^\alpha$ denotes the vector
\[ 
\hat{\bar{h}}^{\alpha} =  \begin{bmatrix} h^{\alpha}_{111} & h^{\alpha}_{222} & h^{\alpha}_{333} & h^{\alpha}_{112} 
& h^{\alpha}_{311} & h^{\alpha}_{223} & h^{\alpha}_{122} & h^{\alpha}_{331} & h^{\alpha}_{233}   \end{bmatrix}^{\sf T} ,  
\]
and $\hat{r}_0$ is the vector obtained by evaluating $G$ on the vector $\hat{h}$ with
\[ 
h^\alpha_{123} = \tfrac{1}{6} (\gamma^\alpha_4 + \gamma^\alpha_5 + \gamma^\alpha_6)\lambda, \qquad \alpha = 4,5,6, 
\]
and all other $h^\alpha_{ijk}$ equal to $0$.  Thus, it suffices to show that the $15 \times 27$ matrix
\[ 
\begin{bmatrix} G^4 & G^5 & G^6 \end{bmatrix} 
\]
has rank 15.

In order to compute the rank of this matrix, observe that equation \eqref{fancy-H} implies that
\[ 
\begin{bmatrix} G^4 & G^5 & G^6 \end{bmatrix}  =[\gamma^\alpha_\beta] \begin{bmatrix} \bar{G}^4 & \bar{G}^5 & \bar{G}^6 \end{bmatrix} , 
\]
where $\bar{G}^\alpha$ represents the matrix $G^\alpha$ with all entries $H^\alpha_{ij}$ replaced by $\bar{H}^\alpha_{ij}$.  
Thus, the rank of the matrix $\begin{bmatrix} G^4 & G^5 & G^6 \end{bmatrix}$ is equal to the rank of the matrix 
$\begin{bmatrix} \bar{G}^4 & \bar{G}^5 & \bar{G}^6 \end{bmatrix}$.  
We can compute this rank explicitly: Substitution of  the expressions \eqref{Hbar-basis} 
for $\bar{H}^4, \bar{H}^5, \bar{H}^6$ into \eqref{G-alpha-matrix} yields
\[
 \bar{G}^4 =
 \begin{bmatrix}
0   &     0   &     0   &     0   &     0   &     0   &     1   &     0   &     0  \\[0.1in]
0   &     0   &     0   &     0   &     0   &     0   &     0   &     1   &     0   \\[0.1in]
0   &     0   &     0   &     0   &     \sigma   &     0   &     0   &     0   &     0   \\[0.1in]
- \tfrac{4}{3}\sigma    &     - \tfrac{1}{3}\sigma    &     - \tfrac{1}{3}\sigma   &     0   &     0   &     0   &     0   &     0   &     0  \\[0.1in]
0   &     0   &     0   &     \sigma   &     0   &     0   &     0   &     0   &     0    \\[0.1in]
0   &     0   &     0   &     0   &     0   &     -2 \sigma   &     0   &     0   &     0  \\[0.1in]
0   &     1   &     0   &     0   &     0   &     0   &     0   &     0   &     0  \\[0.1in]
0   &     0   &     0   &     -\sigma   &     0   &     -1   &     0   &     0   &     0    \\[0.1in]
0   &     0    &     0   &     0   &     0   &     0   &     \sigma   &     0   &     0  \\[0.1in]
\tfrac{1}{3}\sigma^2   &     \tfrac{1}{3}\sigma^2   &     \tfrac{1}{3}\sigma^2   &     0   &     0   &     0   &     0   &     0   &     0   \\[0.1in]
0   &     0   &     1   &     0   &     0   &     0   &     0   &     0   &     0  \\[0.1in]
0   &     0   &     0   &     0   &     0   &     0   &     0   &     0   &     -2 \sigma   \\[0.1in]
\tfrac{1}{3}\sigma^2   &     \tfrac{1}{3}\sigma^2    &     \tfrac{1}{3}\sigma^2    &     0   &     0   &     0   &     0   &     0   &     0 \\[0.1in]
0   &     0   &     0    &     0   &     0   &     0   &     0   &     \sigma   &     0   \\[0.1in]
0   &     0   &     0   &     0   &     -\sigma   &     0   &     0   &     0   &     -1
\end{bmatrix},
\]
\[ \bar{G}^5 =
 \begin{bmatrix}
1 &   0 &   0 &   0 &   0 &   0 &   0 &   0 &   0  \\[0.1in]
0 &   0 &   0 &   0 &   -2 \sigma &   0 &   0 &   0 &   0   \\[0.1in]
\tfrac{1}{3}\sigma^2 &   \tfrac{1}{3}\sigma^2 &   \tfrac{1}{3}\sigma^2 &   0 &   0 &   0 &   0 &   0 &   0   \\[0.1in]
0  &  0  &  0  &   \sigma &   0 &   0 &   0 &   0 &   0  \\[0.1in]
0 &   0 &   0 &   0 &   -1 &   0 &   -\sigma &   0 &   0    \\[0.1in]
0 &   0 &   0 &   0 &   0 &   0 &   0 &   0 &   1  \\[0.1in]
0 &   0 &   0 &   1 &   0 &   0 &   0 &   0 &   0  \\[0.1in]
0 &   0 &   0 &   0 &   0 &  0  &   \sigma &   0 &   0    \\[0.1in]
-\tfrac{1}{3}\sigma  &   -\tfrac{4}{3}\sigma  &   -\tfrac{1}{3}\sigma  &   0 &   0 &   0 &   0 &   0 &   0  \\[0.1in]
0 &   0 &  0 &   0 &   0 &   \sigma &   0 &   0 &   0   \\[0.1in]
0 &   0 &   0 &   0 &   0 &   0 &   0 &   -2 \sigma &   0  \\[0.1in]
0 &   0 &   1 &   0 &   0 &   0 &   0 &   0 &   0   \\[0.1in]
0 &   0  &   0  &   0 &   0 &   -\sigma &   0 &   -1 &   0 \\[0.1in]
0 &   0 &   0  &   0 &   0 &   0 &   0 &   0 &   \sigma   \\[0.1in]
\tfrac{1}{3}\sigma^2 &   \tfrac{1}{3}\sigma^2 &   \tfrac{1}{3}\sigma^2 &   0 &   0 &   0 &   0 &   0 &   0
\end{bmatrix},
\]
\[ \bar{G}^6 =
\begin{bmatrix}
0   &     0   &     0   &     -2 \sigma   &     0   &     0   &     0   &     0   &     0  \\[0.1in]
1   &     0   &     0   &     0   &     0   &     0   &     0   &     0   &     0   \\[0.1in]
0  &   0  &   0  &     -1   &     0   &     0   &     0   &     -\sigma   &     0   \\[0.1in]
0  &   0   &   0   &     0   &     \sigma   &     0   &     0   &     0   &     0  \\[0.1in]
\tfrac{1}{3} \sigma^2   &     \tfrac{1}{3} \sigma^2   &     \tfrac{1}{3} \sigma^2   &     0   &     0   &     0   &     0   &     0   &     0    \\[0.1in]
0   &     1   &     0   &     0   &     0   &     0   &     0   &     0   &     0  \\[0.1in]
0   &     0   &     0   &     0   &     0   &     0   &     -2 \sigma   &     0   &     0  \\[0.1in]
\tfrac{1}{3}\sigma^2   &     \tfrac{1}{3}\sigma^2   &     \tfrac{1}{3}\sigma^2   &    0   &     0   &     0   &     0   &     0   &     0    \\[0.1in]
0 &   0   &    0   &     0   &     0   &     \sigma   &    0   &     0   &     0  \\[0.1in]
0   &    0   &    0   &     0   &     0   &     0   &     -1   &     0   &     -\sigma   \\[0.1in]
0   &     0   &     0   &     0   &     1   &     0   &     0   &     0   &     0  \\[0.1in]
0   &     0   &     0   &     0   &     0   &     1   &     0   &     0   &     0   \\[0.1in]
0   &    0   &    0    &     0   &     0   &    0   &     0   & 0   &     \sigma \\[0.1in]
-\tfrac{1}{3}\sigma    &     -\tfrac{1}{3}\sigma    &      -\tfrac{4}{3}\sigma    &     0   &     0   &     0   &     0   &     0   &     0   \\[0.1in]
0   &    0   &   0   &     0   &     0   &     0   &     0   &     \sigma   &     0
\end{bmatrix}.
\]
Then a direct computation shows that the matrix $\begin{bmatrix} \bar{G}^4 & \bar{G}^5 & \bar{G}^6 \end{bmatrix}$ 
has rank $15$; for example, the submatrix consisting of columns $(2,3,6,7,9,10,12,14,17,18,19,20,22,23,25)$ has determinant equal to
$-\tfrac{64}{27} \sigma^3 (\sigma - 1)^3 (\sigma^2 + \sigma + 1)^2 \neq 0$.
Therefore, $G$ is surjective onto $\mathcal{R}$, and any given values of $r_{ijk\ell,m}$ may be realized by an appropriate choice of $h^\alpha_{ijk}$.

We conclude that, for any $\hat{R} \neq 0$ and any $r_{ijk\ell,m}$, there exist solutions $H^\alpha_{ij}$, $A^{kij}$, $h^\alpha_{ijk}$, 
and $a^{kij}_\ell$ to equations \eqref{Gauss-eqns}--\eqref{d-annihilator-eqns} that satisfy the conditions \eqref{aijkl-condition-3D} 
and \eqref{h-condition-3D}, and hence the linearized system \eqref{equation:unsymmetrized-system} can be transformed 
to a strongly symmetric positive system via a change of variables of the form \eqref{change-of-variables-form}.  
This completes the proof for $n=3$.
\qed

In closing, we note that the strong symmetric positivity condition \eqref{define-Q1} is extremely fragile under changes 
of coordinates, as described in Lemma \ref{Q0-Q1-transformation-lemma}---indeed, this is precisely why
we have to choose local coordinates so carefully in our proof of Theorem \ref{main-ssp-theorem}.  
In future work, we hope to explore this condition in more depth and to obtain a more intrinsic understanding of its significance.

\appendix
\section{Theorems from analysis}\label{classic-theorems-app}

\begin{theorem}[Nash-Moser-Schwartz-Sergeraert]\label{Nash-Moser-theorem}
Let $E_0, F_0$ be real Banach spaces, and let $E_k$ {\rm (}resp. $F_k${\rm )}, $k \in \mathbb{N}$,
be vector subspaces of $E_0$ {\rm (}resp. $F_0${\rm )}, such that $E_{k+1} \subset E_k$ {\rm (}resp., $F_{k+1} \subset F_k${\rm )}.
Let each space $E_k$ {\rm (}resp. $F_k${\rm )} be equipped with a Banach norm $\| \cdot \|_k$ such that
the inclusions $E_{k+1} \hookrightarrow E_k$ {\rm (}resp. $F_{k+1} \hookrightarrow F_k${\rm )} are continuous.
Let $E_\infty = \displaystyle{\bigcap_{k=0}^\infty} E_k$ and $F_\infty = \displaystyle{\bigcap_{k=0}^\infty} F_k$ be given the intersection topology.
Moreover, suppose that there exists a family of linear ``smoothing operators" $S(t): E_0 \to E_\infty$, defined for $t \in \R^+$, satisfying
\begin{equation}\label{smoothing-estimates}
\begin{alignedat}{2}
  & \|\bfu - S(t) \bfu\|_i \leq M_{i,j} t^{i-j} \|\bfu\|_j, & \qquad &  t \in \R^+, i \leq j, \bfu \in E_j;  \qquad \qquad \qquad \qquad \\
  & \|S(t) \bfu\|_j \leq M_{i,j} t^{j-i} \|\bfu\|_i, & \qquad & t \in \R^+, i \leq j, \bfu \in E_i,
\end{alignedat}
\end{equation}
where $M_{ij}$ are positive real constants.

Let $\bfu_0 \in E_\infty${\rm ;} let $D_0 \subset E_0$ be a neighborhood of $\bfu_0$, and let $D_k = D_0 \cap E_k$ for $k \geq 0$.
Let $\Phi: D_0 \to F_0$ be a $C^2$ map, and suppose that there exists an integer $\alpha \geq 0$ satisfying the following assumptions{\rm :}
\begin{enumerate}
\item[\rm (i)] For any $k \geq 0$, $\Phi(D_k) \subset F_k$.
\item[\rm (ii)] There exists a constant $C'$ such that,
for any $\bfu \in D_\alpha$ and $\bfv \in E_\alpha$ such that $\bfu + \bfv \in D_\alpha$,
\begin{equation}\label{Nash-Moser-Phi-bounds}
\begin{gathered}
\| \Phi(\bfu + \bfv)  - \Phi(\bfu) \|_\alpha  \leq C' \|\bfv\|_\alpha , \\
\| \Phi(\bfu + \bfv)  - \Phi(\bfu) - \Phi'(\bfu)\bfv \|_\alpha  \leq C'\|\bfv\|^2_\alpha.
\end{gathered}
\end{equation}
\item[(iii)] There exist constants $C_k > 0$ with the property that, for any $\bfu \in D_\alpha$, 
there exists a continuous linear map $R(\bfu): F_\alpha \to E_0$ such that, for all $\bfh \in F_{\alpha}$,
\[ \Phi'(\bfu)\, R(\bfu)\, \bfh = \bfh, \]
and for all $k \geq 0$, $\bfu \in D_{k+\alpha}$, and $\bfh \in F_{k+\alpha}$,
\begin{equation}
 \|R(\bfu)\,\bfh\|_k \leq C_k(\|\bfh\|_{k+\alpha} + \|\bfh\|_\alpha \|\bfu - \bfu_0\|_{k+\alpha} ). \label{Nash-Moser-estimates}
\end{equation}
\end{enumerate}
Then there exists $\epsilon > 0$ such that, for any $\bff \in F_\infty$ with
\[ \|\bff - \Phi(\bfu_0)\|_\alpha < \epsilon, \]
there exists $\bfu \in D_\infty$ such that
\[ 
\Phi(\bfu) = \bff. 
\]
\end{theorem}

The proof of this theorem can be found in \cite{Schwartz69} and \cite{Sergeraert70}.

\begin{theorem}[Stein]\label{Stein-theorem}
Let $\Omega \subset \R^n$ be a bounded Lipschitz domain.  Then there exists a linear extension operator
\[ \scE: L^1(\Omega) \to L^1(\R^n) \]
satisfying{\rm :}
\begin{enumerate}
\item[\rm (i)] $(\scE f)\vert_\Omega = f$; i.e., $\scE$ is an extension operator.
\item[\rm (ii)]  The restriction of $\scE$ to $W^{k,p}(\Omega)$ is a bounded linear operator
\[ \scE:W^{k,p}(\Omega) \to W^{k,p}(\R^n), \qquad 1 \leq p \leq \infty, \ \ 0 \leq k < \infty. \]
\end{enumerate}
\end{theorem}

The proof of this theorem can be found in \cite{Stein70}.

\section*{Acknowledgments}
The authors gratefully acknowledge the support of a SQuaRE grant from the American Institute of Mathematics, without which this project would not have been possible.
We all thank our late friend Thomas H. Otway for helpful discussions.
G.-Q. Chen was supported in part by the UK Engineering and Physical Sciences Research Council Award EP/L015811/1.
J. Clelland was supported in part by NSF grants DMS-0908456 and DMS-1206272.  
M. Slemrod was supported in part by Simons Collaborative Research Grant 232531 and a Visiting Senior Research Fellowship at Keble College (Oxford).  
D. Wang was supported in part by NSF grants DMS-1312800 and DMS-1613213.
D. Yang was supported in part by NSF grant DMS-1007347.

\nocite{Moser2}
\nocite{Otway12}

\bibliographystyle{amsplain}

\end{document}